\documentclass[reqno,12pt]{amsart}

\usepackage{amsthm,amsfonts,amsmath,mathdots,dsfont,mathrsfs, amssymb,tabulary}
\usepackage[pdftex,bookmarks,colorlinks,pdfmenubar]{hyperref}
\usepackage{ytableau}
\ytableausetup{smalltableaux}

\newcommand{\BF}{{\mathbb {F}}}
\newcommand{\BQ}{{\mathbb {Q}}}
\newcommand{\BZ}{{\mathbb {Z}}}

\newcommand{\CN}{{\mathcal {N}}}
\newcommand{\CL}{{\mathcal {L}}}
\newcommand{\CU}{{\mathcal {U}}}
\newcommand{\CP}{{\mathcal {P}}}
\newcommand{\CS}{{\mathcal {S}}}
\newcommand{\CT}{{\mathcal {T}}}

\newcommand{\RP}{{\mathrm {P}}}
\newcommand{\RU}{{\mathrm {U}}}

\newcommand{\ScE}{{\mathscr {E}}}
\newcommand{\ScS}{{\mathscr {S}}}
\newcommand{\ScT}{{\mathscr {T}}}

\newcommand{\tr}{{\mathrm{tr}}}
\newcommand{\triv}{{\mathds{1}}}
\newcommand{\Ad}{{\mathrm{Ad}}}

\newcommand{\GL}{{\mathrm{GL}}}
\newcommand{\Hom}{{\mathrm{Hom}}}

\newcommand{\GSp}{{\mathrm{GSp}}}
\newcommand{\Sp}{{\mathrm{Sp}}}
\newcommand{\sgn}{{\mathrm{sgn}}}

\newcommand{\Ind}{{\mathrm{Ind}}}

\newcommand{\cpair}[1]{\left\{{#1}\right\}}
\newcommand{\apair}[1]{\left\langle {#1} \right\rangle}

\newcommand{\Tab}{{\rm Tab}}

\newcommand{\ov}{\overline}
\def\bks{{\backslash}}
\def\diag{{\rm diag}}
\def\eps{{\epsilon}}

\def\lam{{\lambda}}
\def\Lam{{\Lambda}}
\def\sig{{\sigma}}

\theoremstyle{plain}
\newtheorem{thm}{Theorem}[section]
\newtheorem{lm}[thm]{Lemma}
\newtheorem{cor}[thm]{Corollary}
\newtheorem{pro}[thm]{Proposition}

\newtheorem{ex}[thm]{Example}

\title{
$\Sp_{2n}(\BF_{q^{2}})$-Invariants In Irreducible Unipotent  Representations of $\Sp_{4n}(\BF_{q})$}

\author{Lei Zhang}
\address{Department of Mathematics\\
Boston College\\
Chestnut Hill, MA 02467-3806}
\email{lei.zhang.2@bc.edu}

\subjclass[2010]{Primary 20G05, Secondary 20C15, 22E50}

\date{}

\keywords{Finite groups of Lie type, Unipotent Representations, Distinguished Representations, Symmetric Subgroups}

\begin{document}
\begin{abstract}
We show that for any irreducible representation of $\Sp_{4n}(\BF_{q})$, the subspace of all its $\Sp_{2n}(\BF_{q^{2}})$-invariants 
is at most one-dimensional.  
In terms of Lusztig symbols, we give a complete list of irreducible unipotent representations of $\Sp_{4n}(\BF_{q})$ which have a nonzero $\Sp_{2n}(\BF_{q^{2}})$-invariant and, in particular, we prove that every irreducible unipotent cuspidal representation has a  one-dimensional subspace of  $\Sp_{2n}(\BF_{q^{2}})$-invariants. As an application, we give an elementary proof of the fact that the unipotent cuspidal 
representation is defined over $\BQ$, which was proved by Lusztig in \cite{Lu02}. 

\end{abstract}
\maketitle
\tableofcontents
\section{Introduction}
Let $G$ be a reductive group defined over a field $k$ and $\theta$ be an involution of $G$ defined over $k$. Denote by $H$ the 
subgroup of all $\theta$-fixed points in $G$, called a {\it symmetric subgroup} of $G$.
One of the classical problems in harmonic analysis on {\it symmetric spaces} $G(k)/H(k)$ is to give a formula for the dimension of
the space $\Hom_{G(k)}(\pi,\Ind^{G(k)}_{H(k)}\triv)$ for each irreducible representation $\pi$ of $G(k)$ and 
the trivial representation $\triv$ of $H(k)$, or equivalently, for the dimension of the space of all $H(k)$-invariant linear functionals on the 
space of $\pi$. 
By Frobenius reciprocity, this dimension is equal to  the dimension $\dim \pi^{H(k)}$ of the subspace of  $H(k)$-invariants.

For example, let $k$ be a non-archimedean field of characteristic 0 and $\pi$ be an irreducible supercuspidal representation of $G(k)$ constructed by J.-K. Yu and J.-L. Kim. In \cite{HM08}, Hakim and Murnaghan  give a formula for the dimension of the space of such invariant linear functionals for all symmetric spaces. This formula reduces the computation of the dimension formula to the ``level zero part.'' For depth-zero supercuspidal representations, their formula reduces to the analogous problem over the residue field of $k$.
This is one motivation  to study such problems over finite fields.

Over a finite field $\BF_{q}$ of odd characteristic $p$,  by convention throughout this paper, we use $F$ to denote the 
Frobenius map  $F\colon G\to G$, corresponding to the $\BF_{q}$-structure of $G$, and
we use the inner product $\langle,\rangle_{G^{F}}$ to denote the dimension of the $G^{F}$-invariant $\Hom$-space.
 Such dimension formulas for symmetric spaces have been studied by Lusztig \cite{Lu90,Lu00}, Henderson \cite{Hen03} and others.
In \cite{Lu90}, Lusztig gives a dimensional formula for all Deligne-Lusztig virtual representations $R^{\lam}_{T}$
and  for all symmetric subgroups $H$, that is, for $\langle R^{\lam}_{T},\Ind^{G^{F}}_{H^{F}}\triv\rangle _{G^{F}}$. 
Further, applying this formula for $R^{\lam}_{T}$, Lusztig gives a dimension formula in the case when $G^{F^{2}}$ and $H=G^{F}$ in \cite{Lu00}, and Henderson \cite{Hen03} solves the problem in the case when $G=\GL$ and $H$ is any symmetric subgroup.

The objective of this paper is to consider the case that the symplectic group $G^{F}=\Sp_{4n}(\BF_q)$ and its symmetric subgroup 
$H^{F}=\Sp_{2n}(\BF_{q^{2}})$. 
Let $J_{2n}$ be the skew-symmetric matrix given by
$$
J_{2n}=\begin{pmatrix}
&w_{n}\\-w_{n}&
\end{pmatrix},
$$
where $w_{n}$ is the $n$-by-$n$ permutation matrix with the unit anti-diagonal. Take $G\subset \GL_{4n}$ to be  the  group preserving the symplectic form defined by $J_{4n}$.
Let $\theta$ be the adjoint map $\Ad(\varepsilon)$, where $\varepsilon$ is given by
$$
\varepsilon=\begin{pmatrix}
&I_{n}&&\\ \tau I_{n}&&&\\&&&I_{n}\\&&\tau I_{n}&
\end{pmatrix}
$$
and $\tau$ is a primitive element in $\BF_{q}$.
Since $\varepsilon^{2}$ is in the center of the symplectic similitude group $\GSp^{F}_{4n}$, $\theta$ is an involution of $G^{F}$ and
the symmetric subgroup $H$ is isomorphic to ${\rm Res}_{\BF_{q^{2}}/\BF_{q}}\Sp_{2n}$, i.e. $H^{F}=\Sp_{2n}(\BF_{q^{2}})$ where $\BF_{q^{2}}$ is the quadratic extension $\BF_{q}(\sqrt{\tau})$ of $\BF_{q}$.

First, we give an upper bound for the dimension of the subspace of $\Sp_{2n}(\BF_{q^{2}})$-invariants in this case.
In \cite{zhang10}, we prove that the symmetric pair $(G(k),H(k))$ defined over non-archimedean field $k$ is a {\it Gelfand pair},  that is,
$$
\dim\Hom_{G(k)}(\pi,\Ind^{G(k)}_{H(k)}\triv)\leq 1
$$
for all irreducible smooth representations of $G(k)$.
Using the calculation in \cite{zhang10}, we have
 \begin{thm} \label{thm:uniq}
For each irreducible representation $\pi$ of $\Sp_{4n}(\BF_q)$,
$$
\langle{\pi, \Ind^{\Sp_{4n}(\BF_q)}_{\Sp_{2n}(\BF_{q^{2}})}\triv}\rangle\leq 1.
$$
\end{thm}
Indeed, this can be deduced from $H^{F}gH^{F}=H^{F}g^{-1}H^{F}$ for all $g\in G^{F}$ over the finite field $\BF_{q}$ 
and the calculation whose analogue over non-archimedean local fields was carried out \cite{zhang10}. 
It is easy to check that such a calculation holds for a finite field $\BF_{q}$ of odd characteristic as well.
We will omit the details in this paper.

Our goal is to give a complete list of irreducible representations of $G^F$ which have a non-zero $H^{F}$-invariant, which are simply called the 
{\it distinguished representations}.
In this paper, we will focus on {\it unipotent representations} of $G^{F}$, which occur as components of a Deligne-Lusztig virtual representation $R^{\triv}_{T}$, where $T$ is an $F$-stable maximal torus of $G$ and $\triv$ is its trivial character.
For a general exposition, see Lusztig \cite{Lu81, Lu84}.

In this case,  unipotent representations are parametrized by symbols of rank $n$ and odd defect.
In detail,
if $G$ is $F$-split, the $F$-stable maximal tori of $G^{F}$ can be parametrized by conjugacy classes of the Weyl group $W$ of $G$, we use $R^{\triv}_{T_{w}}$ to denote a Deligne-Lusztig virtual representation.
First, unipotent representations can be divided into families by
\begin{equation}\label{eq:c}
R(c)=\frac{1}{|W|}\tr(w,c)R^{\triv}_{T_{w}}.
\end{equation}
Here $c$ is a {\it cell} defined by Lusztig, which is a virtual representation of $W$.
In addition, $R(c)$ is an actual representation of $G^{F}$ (not necessarily irreducible), with the smallest number of irreducible 
constituents one can expect from a linear combination of $R^{\triv}_{T_{w}}$'s.
In the $\GL_{n}$ case, a cell $c$ is  just an irreducible module of the symmetric group $S_{n}$ and $R(c)$ is an irreducible representation of $G^{F}$.
There is a one-to-one correspondence between irreducible unipotent representations of $\GL^{F}_{n}$ and irreducible $S_{n}$-modules.
However,  this is not true for the other finite groups of Lie type.
In the symplectic and odd orthogonal group cases, the irreducible modules of Weyl groups are parametrized by symbols of defect 1 as a subset of symbols of odd defect.

Using Lusztig's formula in \cite{Lu90}, we have a dimension formula for $R^{\triv}_{T}$ and then obtain a $W$-module $\Xi$ such that for $w\in W$
$$
\tr(w,\Xi)=\langle R^{\triv}_{T_{w}}, \Ind^{G^{F}}_{H^{F}}\triv \rangle.
$$
If $G=\GL$, $\Xi$ is always an actual representation of $S_{n}$ and it is enough to decompose the module $\Xi$ into irreducibles.
Once one has the multiplicities of the irreducibles, the multiplicity of the corresponding unipotent representation  of $G$ follows.
From the above discussion, this is not enough for the other reductive groups and creates  difficulties in our case 
in comparison to the general linear group case.
The issue is that even if we know the dimension formula for $R^{\lam}_{T}$, there is still more work to be done 
in order to obtain the dimension formula for all irreducible unipotent representations.
In addition to this issue, the $W$-module $\Xi$ is usually {\it not} an actual representation in our case and 
it requires extra work here to decompose $\Xi$ into irreducibles.

Define $\CU(\Ind^{G^{F}}_{H^{F}}\triv)$ to be the sub-representation of $\Ind^{G^{F}}_{H^{F}}\triv$ whose irreducible sub-representations are unipotent representations of $G^{F}$. 
Our purpose is to decompose the representation $\CU(\Ind^{G^{F}}_{H^{F}}\triv)$ into irreducible representations of $G^{F}$. 
The strategy is to decompose the virtual representation $\Xi_{n}$ of $W_{2n}$, the Weyl group of $\Sp_{4n}(\BF_{q})$, as
$$
\Xi_{n}=\sum_{r=0}^{n}\sum_{\beta\vdash n-r}\sum^{r}_{i=1}
(-1)^{i}\binom{(i,2r-i)\cdot \beta}{\beta}
$$
where pairs of partitions $\binom{\alpha}{\beta}$ parametrize irreducible representations of $W_{2n}$ and $(i,2r-i)\cdot \beta$ corresponds to an induced product character of $S_{2r}\times S_{2n-2r}$.
By the Littlewood-Richardson rule and combinatorial arguments, we divide the irreducibles of $\Xi_{n}$ into a family of cells $c(Z,\Phi_{Z},\hat{\Phi}_{Z})$ such that
$$
\langle \Xi_{n}, c(Z,\Phi_{Z},\hat{\Phi}_{Z})\rangle=2^{d}
$$
where $d$ is the number of singles in the special symbol $Z$. By the definition of cells, the associated representation $R(c(Z,\Phi_{Z},\hat{\Phi}_{Z}))$ of $G^{F}$ has exactly $2^{d}$ constituents. Since the multiplicity of each constituent is at most one, every constituent is distinguished in this cell. Therefore, all unipotent representations parametrized by symbols in $c(Z,\Phi_{Z},\hat{\Phi}_{Z})$ are distinguished and hence we have
\begin{thm}\label{thm:list}
$$
\CU(\Ind^{\Sp_{4n}(\BF_{q})}_{\Sp_{2n}(\BF_{q^{2}})}\triv)=\sum_{Z\in  \CS_{n}} R(c(Z,\Phi_{Z},\hat{\Phi}_{Z})).
$$
In particular, the unipotent cuspidal representation of $\Sp_{4n}(\BF_{q})$
has  nontrivial $\Sp_{2n}(\BF_{q^{2}})$-invariants.
\end{thm}

This result has applications to the theory of irreducible admissible representations of $\Sp_{4n}(k)$ over a non-archimedean field $k$ 
and to the theory of automorphic representations of $\Sp_{4n}$ over number fields. In \cite{zhang13S}, we 
apply Theorem \ref{thm:list} and construct a family of distinguished representations of $\Sp_{4n}(k)$. Further applications including 
the determination of associated number-theoretic invariants such as $L$-functions and Arthur parameters and the construction 
of cuspidal automorphic forms which have a non-zero period integral associated to the symmetric pair will be considered in our 
forth-coming work. 

In this paper, we give an application of Theorem \ref{thm:list} to prove the rationality of unipotent representations 
of symplectic groups. This extends  the work of Lusztig \cite[Section 3]{Lu02} to all unipotent cuspidal representations for symplectic groups.
We remark that our proof takes an elementary approach without using the Hasse principle.
By Theorem~\ref{thm:list}, the unipotent cuspidal representation of $G^{F}$ is multiplicity-free in the $\BQ[G^{F}]$-module $\Ind^{G^{F}}_{H^{F}}\triv$.
Since the characters of unipotent representations are $\BQ$-valued,  the unipotent cuspidal representation of $G^{F}$ can be realized by a $\BQ$-module. Hence we obtain 
\begin{cor}\label{cor:ration}
All unipotent cuspidal representations of symplectic groups over finite fields of odd characteristic are defined over $\BQ$.
\end{cor}

The case of  non-unipotent distinguished representations  of $G^{F}$ is also of interest. 
This will be considered in a sequel to this paper. 
In that case we need to deal with  unipotent representations of {\it even orthogonal groups} and {\it general linear groups}. 
Since the center of symplectic groups is {\it disconnected}, many arguments given here may be substantially changed.

This paper is organized as follows. In Section~\ref{sec:uni}, for the benefit of the reader, we  recall
Lusztig's classification theory of unipotent representations of symplectic groups (see Lusztig \cite{Lu81} or Carter \cite{Car93}), and Lusztig's general formula for $R^{\lam}_{T}$ in \cite{Lu90}.
 In Section~\ref{sec:Xi}, we decompose the virtual $W_{2n}$-module $\Xi_{2n}$ into well-understood $W_{2n}$-modules. In Section~\ref{sec:main}, we give an explicit combinatorial description of unipotent representations in terms of Lusztig symbols.

{\it Acknowledgement.} I thank Professor George Lusztig for comments and suggestions. He suggested that this classification of the distinguished representations could be applied to prove the rationality of unipotent cuspidal representations, which is our Corollary~\ref{cor:ration}.
I also thank Professors Dihua Jiang, Fiona Murnaghan, Solomon Friedberg and C. Ryan Vinroot for their useful communications and comments on topics related to this paper.
The research in this paper was motivated by my Ph.D. thesis project in University of Minnesota under the supervision of Professor Dihua Jiang.

{\bf Notation.}
Denote by $S_{n}$ the symmetric group of degree $n$.
We also consider $w_{n}$ as  a longest Weyl element of $S_{n}$.

Let $\alpha$ be a partition of $n$ with parts $\alpha_{1},\ \alpha_{2},\dots,\ \alpha_{m}$.
 Its size, length, and transpose are denoted by $|\alpha|$, $\ell(\alpha)$, and $\alpha'$.
 The multiplicity of $i$ as a part of $\alpha$ is denoted by $m_{i}(\alpha)$ or  $m_{i}$.

If $G$ is a group and $g\in G$, denote by $Cl_{G}(g)$ the conjugacy class of $g$ in $G$ and $Z_{G}(g)$ the centralizer of $g$ in $G$.
Let $\BZ_{n}$ be a cyclic group of order $n$.

\section{Unipotent Representations of the Symplectic Group}  \label{sec:uni}

In this section, we shall recall the Deligne-Lusztig  representations in \cite{DL76}  and Lusztig's classification of irreducible representations of symplectic groups over  finite fields.

Let $G$ be  a connected reductive group defined over a finite field $\BF_q$ of odd characteristic $p$ and $T$ be an $F$-stable maximal torus of $G$.
Let $\lambda$ be  a character  of $T^F$ over $\bar{\BQ}^{\times}_l$, where $\bar{\BQ}_{l}$ is an algebraic closure of $\BQ_l$ and $l$ is prime to $q$.
For each pair $(T,\lambda)$, Deligne and Lusztig \cite{DL76} attached a virtual representation $R^{\lam}_T$ of $G^F$.

First, there is a bijection between geometric conjugacy classes (see Definition 5.5 in \cite{DL76}) of the pairs $(T,\lambda)$ and semi-simple conjugacy classes in ${G^{*}}^{F}$.
Here  $G^{*}$ is the dual group of $G$  defined over $\BF_{q}$.
Further, referring to \cite[Section 7.5]{Lu77}, one has a bijection between $G^{F}$-conjugacy classes $(T,\lambda)$ and ${G^{*}}^{F}$-conjugacy classes $(T',s)$, with $T'$ an $F$-stable maximal torus in $G^{*}$ and $s\in T'^{F}$. We also use $R^{s}_{T}$ to denote $R^{\lam}_{T}$ for the corresponding pair $(T,\lam)$.

Let $\ScE(G)$ be the set of isomorphism classes of irreducible representations of $G^F$ over $\bar{\BQ}_l$. Let $\ScE(G,s)$ be the subset of $\ScE(G)$, consisting of all irreducible representations $\rho$ such that $\apair{\rho,R^{s}_{T}}\neq 0$ for some $T$.
One has a partition
 $$\ScE(G)=\coprod_{s}\ScE(G,s)$$
  where $s$ runs through the set of semi-simple ${G^{*}}^{F}$-conjugacy classes in ${G^{*}}^{F}$.

An irreducible representation $\rho$ of $G^{F}$ is \emph{unipotent} if $\apair{\rho,R^{\triv}_{T}}_{G^{F}}\neq 0$ for some $F$-stable maximal  torus $T$, that is, $\rho$ is in $\ScE(G,1)$.
In \cite{Lu84}, Lusztig parametrized all the unipotent representations of reductive groups. We refer to  the parametrization of classical groups in Lusztig \cite{Lu77}.

\subsection{Unipotent Representations of Symplectic Groups}

In this section, we recall the classification of  unipotent  representations of symplectic groups
and  Lusztig symbols. See Lusztig \cite{Lu81} or Carter \cite{Car93}.
A \emph{symbol} is an unordered pair $\Lam=\left(\begin{smallmatrix}S\\T\end{smallmatrix}\right)$, where $S$ and $T$ are finite sets consisting of non-negative integers.
In this paper, the symbol is reduced, i.e. $0\notin S\cap T$.
The \emph{rank} $\mathrm{rk}(\Lam)$ of $\Lam$ is  defined by
$$
\mathrm{rk}(\Lam)=\sum_{\lam\in S}\lam+\sum_{\mu\in T}\mu-\left[\left(\frac{\#S+\#T-1}{2}\right)^2\right],
$$
where the square bracket is the greatest integer  function.
Define by the \emph{defect} $\mathrm{def}(\Lam)=|\#S-\#T|$.
 Obviously, the rank of a symbol is a non-negative integer and $rk(\Lambda) \geq \left[\left(\mathrm{def}(\Lam)/2\right)^2\right]$. 
Lusztig used symbols to parametrize all unipotent representations of $G^{F}$ as follows. 
\begin{pro}[{\cite[Theorem 8.2]{Lu77}}] \label{thm:uni_sym}
For  the symplectic group, there exists a one-to-one correspondence between the unipotent representations and the set of symbol classes of rank $n$ and odd defect. Moreover, the symbol $\Lam$ with $\mathrm{rk}\Lam=[(\mathrm{def}(\Lam)/2)^{2}]$ corresponds to the cuspidal unipotent representation.
\end{pro}
For example,
the unipotent cuspidal representation of $\Sp_{2(d^{2}+d)}(\BF_{q})$ corresponds to the symbol
$$
\begin{pmatrix}
0,1,2,\dots,2d\\
-
\end{pmatrix}.
$$
For each symbol $\Lam$ of rank $n$ and odd defect we denote by $\rho(\Lam)$ the corresponding unipotent representation of $G^{F}$.

A symbol
$$Z=\begin{pmatrix}z_0, z_2,\cdots,z_{2m}\\
z_1,z_3,\cdots,z_{2m-1}
\end{pmatrix}$$
 of rank $n$ and defect 1 is called \emph{special symbol} if $z_0\le z_1\le z_2\le\cdots\le z_{2m-1}\le z_{2m}$ holds.
 Assume that the number of singles in $Z$ (the elements which occur only once) is $2d+1$. Let $\Phi$ be an arrangement of the $2d+1$ singles in $Z$ into $d$ pairs and one isolated element. An arrangement is called {\it admissible} if it satisfies the following condition: there is a pair consisting of consecutive $z$'s and then the  arrangement is still admissible for the new special symbol so-obtained by  removing this pair.

Let $\Phi$ be an admissible arrangement of the special symbol $Z$. If $\Psi$ is a subset of $\Phi$, denote by $\Psi^{*}$ the set of $z$'s in $\Psi$ occurring in the first row of $Z$ and $\Psi_{*}$ the set of $z$'s in $\Psi$ occurring in the second row of $Z$. Let $Z_{0}$ be the set of $z$'s in $Z$ occurring twice.
For any subset $\hat{\Phi}$ of $\Phi$, define a {\it virtual cell} of $W_{n}$, the Weyl group of $\Sp_{2n}(\BF_{q})$, by
$$
c(Z,\Phi,\hat{\Phi})=\sum_{\Psi\subset \Phi}(-1)^{|\hat{\Phi}\cap \Psi|}\begin{pmatrix}
Z_{0}\coprod \Psi_{*}\coprod (\Phi-\Psi)^{*}\\
Z_{0}\coprod \Psi^{*}\coprod (\Phi-\Psi)_{*}
\end{pmatrix},
$$
where $\Psi$ runs through over all subsets of $\Phi$.
This virtual cell $c(Z,\Phi,\hat{\Phi})$ is considered as a virtual representation of $W_{n}$. 
Indeed, each term in $c(Z,\Phi,\hat{\Phi})$ is a symbol of rank $n$ and defect 1 and not special except $\Psi=\emptyset$. 
There is a one-to-one correspondence between the symbol classes of rank $n$ and defect 1 and the irreducible representations of $W_{n}$ referring to Section~\ref{sec:weyl}. Those symbols of defect 1 are also considered as irreducible representations of $W_{n}$.

Similar to~\eqref{eq:c}, we have a unipotent representation $R(c(Z,\Phi,\hat{\Phi}))$ of $G^{F}$
$$
R(c(Z,\Phi,\hat{\Phi}))=\frac{1}{|W_{2n}|}\tr(w,c(Z,\Phi,\hat{\Phi}))R^{\triv}_{T_{w}}.
$$
 Referring to Lusztig's classification theory, we have
$$
R(c(Z,\Phi,\hat{\Phi}))=\sum^{2^{d}}_{i=1}\rho(\Lam_{i}),
$$
where $\Lam_{i}$ is uniquely determined by the virtual cell $c(Z,\Phi,\hat{\Phi})$. More details may be found in Lusztig \cite{Lu81}.

\begin{ex}\label{ex:Sp4}
Let us give an example of unipotent representations of $\Sp_{4}(\BF_{q})$.
All symbol classes of rank 2 and odd defect are
$$
\left\{ \begin{pmatrix} 2\\-\end{pmatrix},
\begin{pmatrix} 1,2\\0\end{pmatrix},
\begin{pmatrix} 0,2\\1\end{pmatrix},
\begin{pmatrix} 0,1\\2\end{pmatrix},
\begin{pmatrix} 0,1,2\\1,2\end{pmatrix}
\right\}
\text{ and }
\left\{\begin{pmatrix} 0,1,2\\-\end{pmatrix}\right\}.
$$
The symbols in the first set are of defect 1 and also  corresponding to all irreducible representations of $W_{2}$. The symbol in the second set is of defect 3 and corresponding to the unipotent cuspidal representation, that is, $\theta_{10}$ in Srinivasan's notation \cite{Sri68}.
Further,
$$
R\binom{2}{-} \text{ and }
R\binom{0,1,2}{1,2}
$$
are irreducible unipotent irreducible representations of $\Sp_{4}(\BF)$ and the first one is the trivial representation. For the special symbol $Z=\binom{0,2}{1}$, $\Phi_{1}=\{(0,1)\}$ and $\Phi_{2}=\{(1,2)\}$ are the only two admissible arrangements and for instance
\begin{equation}\label{eq:Sp4}
R(c(Z,\Phi_{1},\Phi_{1}))=R\binom{0,2}{1}-R\binom{1,2}{0}=\rho\binom{0,1}{2}+\theta_{10}.
\end{equation}
For additional complementary examples, see Carter \cite[Chapter 13]{Car93}.
\end{ex}

\begin{ex}\label{ex:Sp8}
Let $G^{F}=\Sp_{12}(\BF_{q})$ and $Z=\binom{0,2,4}{1,3}$ be a special symbol whose number of singles is 5. Then $\Phi=\{(0,1),(2,3)\}$ is an admissible arrangement. We have
\begin{equation}\label{eq:Sp12}
R(c(Z,\Phi,\Phi))=
\rho\binom{2,3,4}{0,1}+
\rho\binom{0,3,4}{1,2}+
\rho\binom{0,1,2,3}{4}+
\rho\binom{0,1,2,3,4}{-}.
\end{equation}
\end{ex}

\subsection{Irreducible Representations of the Weyl Group $W_{n}$} \label{sec:weyl}

In this section, we recall the construction of irreducible representations of $W_{n}$ in \cite[\S 2]{Lu81}.

Let $S_{n}$ be the symmetric group of degree $n$ acting on the set
 $
 \ScS_{n}=\{1,2, \cdots, n\}.
 $  
 Let $\alpha$ be a partition of $n$.
Denote by $\rho_{S}(\alpha)$ the irreducible module of $S_{n}$ whose restriction to $S_{\alpha_1} \times S_{\alpha_2} \times \cdots S_{\alpha_{\ell(\alpha)}} \subset  S_n$  contains the trivial character and whose restriction to $S_{\alpha'_1} \times S_{\alpha'_2} \times \cdots\times S_{\alpha'_{\ell(\alpha')}} \subset  S_n$ contains the Steinberg character.
All irreducible modules of $S_{n}$ are parameterized by partitions of $n$ in this way. For example, $\rho_{S}(n)$ is the trivial character and $\rho(1^{n})$ is the Steinberg character.

Let  $W_n$ act on the set
$$\ScT _n=\{1,2, \cdots, n, n', \cdots, 2', 1'\},$$
and be generated by $(i,i+1)(i',(i+1)')$ for $1\leq i\leq n-1$ and by $(n,n')$.
Let $\CP_{n}$ be the set of all pairs of partitions $(\alpha;\beta)$ such that $|\alpha|+|\beta|=n$. Then irreducible modules of $W_{n}$ are parametrized by the pairs in $\CP_{n}$.
Indeed, let $(\alpha;\beta)$ be in $\CP_{n}$,
 and $\rho_{S}(\alpha)$ and $\rho_{S}(\beta)$ be the irreducible representations of $S_{|\alpha|}$ and $S_{|\beta|}$.
 Since $W_m$ is isomorphic to $S_{m}\ltimes (\BZ_{2})^{m}$,  there are natural liftings $\ov{\rho}_S(\alpha)$ and $\ov{\rho}_S(\beta)$ as representations of $W_{|\alpha|}$ and $W_{|\beta|}$.
  Define a quadratic character $\chi$ of $W_n$ via
 \begin{equation*}
 \chi_n(w)=(-1)^{\#\{w(1),w(2),\cdots,w(n)\} \cap\{1',2',\cdots,n'\}}.
 \end{equation*}
 Then the induced representation
 $$
 \binom{\alpha}{\beta}:=\Ind^{W_n}_{W_{|\alpha|}\times W_{|\beta|}}\ov{\rho}_S(\alpha)\otimes(\ov{\rho}_S(\beta)\otimes \chi_{|\beta|})
 $$ is an irreducible representation of $W_n$   corresponding to the ordered pair $(\alpha;\beta)$.

By adding zeroes, we can increase the lengths of $\alpha $ and $\beta$ at  will  and order $0\leq \alpha_{i}\leq \alpha_{i+1}$ for $1\leq i\leq m+1$ and $0\leq \beta_{i}\leq \beta_{i+1}$ for $1\leq i\leq m$.
By convention, at least one of $\alpha_{1}$ and $\beta_{1}$ is not zero.
 Set $\lambda_i=\alpha_i+i-1$ for $1\leq i \leq m+1$ and  $\mu_i=\beta_i +i -1$ for $1\leq i \leq m$ and then we obtain a symbol of rank $n$  and defect 1
\begin{equation*}
\CL\binom{\alpha}{\beta}:=\begin{pmatrix}
\lambda_1,\lam_2,\cdots,\lam_{m+1}\\
\mu_1,\mu_2,\cdots,\mu_m
\end{pmatrix}.
\end{equation*}
Thus  we  have a one-to-one correspondence between irreducible representations of $W_n$ and symbols  of rank $n$ and defect 1.

\subsection{Lusztig's Formula for $\apair{\tr(\cdot , R^{\lam}_T),\Ind^{G}_H \mathds{1}}$}

For the virtual representations $R^{\lam}_{T}$, Lusztig gave a formula in \cite[Theorem 3.3]{Lu90} to compute $\apair{\tr(\cdot,R^{\lam}_T),\Ind^{G^F}_{H^F}\triv}$.
We recall the formula in this section. Define
$$\Theta_T=\{f\in G \mid \theta(f^{-1}Tf)=f^{-1}Tf\}.$$
 Then $T$ (resp.\ $H$) acts on $\Theta_T$ by left (resp.\ right) multiplication. The double cosets $T\backslash \Theta_T /H$ are in one-to-one correspondence with the double cosets $B\bks G/H$, where $B$ is a Borel subgroup containing $T$. 
 Let $\Theta^F_T$ be the set of $F$-fixed elements. The double cosets $T^F\backslash \Theta^F_T/H^F$ are also bijective to the $G^F$-conjugacy classes of $(\theta,F)$-stable maximal tori of $G^F$.

For any $f\in \Theta^F_T$, define a morphism $\epsilon_{T,\lam}$ of $(T\cap fHf^{-1})^F$ by
\begin{equation} \label{eq:eps}
\epsilon_{T,\lam}(t)=(-1)^{\BF_q\textrm{-rank}(Z_G((T\cap fHf^{-1})^{\circ}))+\BF_q\textrm{-rank}(Z^{\circ}_G(t)\cap Z_G((T\cap fHf^{-1})^{\circ}))}.
\end{equation}
Then $\epsilon_{T,\lam}$ is a character and trivial on $((T\cap fHf^{-1})^{\circ})^F$. Define
\begin{equation*}
\Theta^F_{T,\lam}=\{f\in \Theta^F_T\mid\lam_{(T\cap fHf^{-1})^F}=\eps_{T,\lam}\}.
\end{equation*}
The groups $T^F$ and $H^F$ still act on $\Theta^F_{T,\lam}$ by left and right multiplication.
\begin{pro}[{\cite[Theorem 3.3]{Lu90}}] \label{pro:Lusztig}
\begin{equation*}
\apair{\tr(\cdot , R^{\lam}_T),\Ind^{G}_H \mathds{1}}=
\sum_{f\in T^F\bks \Theta^F_{T,\lam}/H^F}(-1)^{\BF_q\text{-rank}(T)+\BF_q\text{-rank}(Z_G((T\cap fHf^{-1})^{\circ}))}
\end{equation*}
\end{pro}
Remark that this formula holds for all connected reductive groups  and its symmetric subgroups.

\section{Decompositions of  $W_{2n}$-module $\Xi_{n}$}\label{sec:Xi}

\subsection{Double Cosets $T\bks \Theta/H$}

Let $G$ be the symplectic group $\Sp_{4n}$.
Fix $T$ to be the set of diagonal matrices of $G$, and $B$ to be the set of upper triangular matrices of $G$. Then $B$ is an $F$-stable Borel subgroup of $G$ containing $T$.
 Let $g$ be an element of $G$ such that $g^{-1}F(g)\in N_{G}(T)$. Then $gTg^{-1}$ is also an $F$-stable maximal torus of $G$.
 As $G$ is $F$-split,
the map from $gTg^{-1}$ to the image of $g^{-1}F(g)$ in $W$ gives a bijection  between the $G^{F}$-conjugacy classes of $F$-stable maximal tori of $G$ and the conjugacy classes of $W$.
Let $g_{w}$ be an element in $G$ such $g^{-1}_{w}F(g_{w})=w$. Denote by $T_{w}=g_{w}Tg^{-1}_{w}$ the corresponding $F$-stable maximal torus.

Let $\vartheta$ be the map on $G$ defined by $\vartheta(g)=g\theta(g^{-1})$.
Then $\vartheta$ gives a bijection from $G/H$ to the image $\Im(\vartheta)$ of $\vartheta$. Note that
$$
\Im(\vartheta)\subset\{g\in G\mid \theta(g)=g^{-1}\}.
$$
Recall that $T$ is a $\theta$-stable maximal torus of $G$.
Let $N(T)$ be the normalizer subgroup of $T$ in $G$ and $\CN=N(T)\cap \Im(\vartheta)$.  Define a {\it $\theta$-twisted action} of $T$ on $\CN$ via $t* \omega= t\omega\theta(t)^{-1}$. This $\theta$-twisted action also induces an action of $T$ on $N(T)/T$, which is the $\varepsilon$-conjugation.
According to \cite[Proposition 6.8]{HW93}, $G$ is the disjoint union of the double cosets $B\gamma H$, where $T*\vartheta(\gamma)$ runs through all the obits of  the action of $T$ on $\CN$.

In our case,  the Weyl group $N(T)/T$ is  $W_{2n}$.
For $x\in N(T)$, decompose $x$ as  $x=tv\varepsilon'$ where $t\in T$ and $v\in W_{2n}$ (consider $v$ as a representative in $N(T)$) and  
$$
\varepsilon'=\begin{pmatrix} &I_{n}&&\\ \tau I_{n}&&&\\&&&\tau^{-1}I_{n}\\&&I_{n}&\end{pmatrix}\in \Sp_{4n}(\BF_{q}).
$$
Then $\theta(x)=x^{-1}$ if and only if
\begin{equation}  \label{eq:x-N}
t\cdot v tv^{-1}\cdot\diag\cpair{\tau I_{2n}, I_{2n}}v \diag\cpair{I_{2n}, \tau^{-1} I_{2n}}=v^{-1}.
\end{equation}
Moreover, $v$ is of order 2 in $W_{2n}$.

 Let $v$ be an element of $W_{2n}$ of order 2.
 Denote by
 \begin{align*}
&\ScS^{(1)}_{v}=\cpair{i\in\ScS_{2n}\mid v(i)\in \ScS_{2n}, v(i)>i},\\
&\ScS^{(2)}_{v}=\cpair{i\in \ScS_{2n}\mid v(i)=j', v(j)=i', i<j},\\
&\ScS^{(3)}_{v}=\cpair{i\in \ScS_{2n}\mid v(i)=i},\\
&\ScS^{(4)}_{v}=\cpair{i\in \ScS_{2n}\mid v(i)=i'}.
\end{align*}
Note that the sets $\ScS^{(1)}_{v}$, $v(\ScS^{(1)}_{v})$, $\ScS^{(2)}_{v}$, $v(\ScS^{(2)}_{v})'$, $\ScS^{(3)}_{v}$, and  $\ScS^{(4)}_{v}$ are a partition of $\ScS_{2n}$, where $i''=i$. Then, $v$ is of form
 $$
 \prod_{a_{i}\in \ScS^{(1)}_{v}}(a_{i},v(a_{i}))(a'_{i},v(a'_{i}))\prod_{b_{i}\in \ScS^{(2)}_{v}}(b_{i},v(b_{i}))(v(b_{i}'),b'_{i}) \prod_{c_{i}\in \ScS^{(4)}_{v}}(c_{i},c'_{i}).
 $$
Denote by $a_{v}$, $b_{v}$ and $c_{v}$ the corresponding product over $\ScS^{(i)}$ for $i=1,2,4$, respectively. Note that each factor  commutes with the others.
We choose the following representatives in $G^{F}$  corresponding to the cycles $(i,j)(i',j')$ and $(i,j)(i',j')(i,i')(j,j')$ respectively by embedding the $4\times 4$ matrices
$$
\begin{pmatrix}
&1&&\\1&&&\\&&&1\\&&1&
\end{pmatrix}
\text{ and }
\begin{pmatrix}
&&1&\\&&&-1\\1&&&\\&-1&&
\end{pmatrix},
$$
 into the $4n\times 4n$ matrices in the $(i,j,j',i')$-th rows and $(i,j,j',i')$-th columns, with 1 for the remaining diagonal entries, and zeros for all the remaining non-diagonal entries.

Denote by $\Omega$ the set of all pairs $(t_{v},v)$ in $T\times W$ satisfying the following conditions:
\begin{itemize}
\item $\ScS^{(4)}_{v}=\emptyset$;
\item $(t_{v})_{i}=\begin{cases}
\tau^{-1} & i\in v(\ScS^{(1)}_{v})\\
\pm 1/\sqrt{\tau} & i\in \ScS^{(3)}_{v}\\
1 & i\in \ScS_{2n}\smallsetminus (v(\ScS^{(1)}_{v})\cup \ScS^{(3)}_{v});
\end{cases}$
\item $m_{v}=\frac{1}{2}\ScS^{(3)}_{v}$, where
\begin{equation}
m_{v}=\#\{i\in \ScS^{(3)}_{v}\mid (t_{v})_{i}=1/\sqrt{\tau}\}.
\end{equation}
\end{itemize}

Let $\gamma(t_{v},v)$ be an element in $G^{F}$ such that $\vartheta(\gamma(t_{v},v))=t_{v}v\varepsilon'$.
Indeed, there exists an element $g_{1}$ in $\GL_{4n}$ such that $g_{1}\varepsilon g_{1}^{-1}=t_{v}v\varepsilon'\varepsilon$, since both $\varepsilon$ and $t_{v}v\varepsilon'\varepsilon$ are $\GL_{4n}$-conjugate to $\diag\{\sqrt{\tau}I_{2n},-\sqrt{\tau}I_{2n}\}$.
Since $\varepsilon\in \GSp_{4n}$, there exists an element $\gamma(t_{v},v)$ in $\Sp_{4n}$ such that $\gamma(t_{v},v)\varepsilon \gamma(t_{v},v)^{-1}=t_{v}v\varepsilon'\varepsilon$ and  then $\vartheta(\gamma(t_{v},v))=t_{v}v\varepsilon'$. In addition, $\{t_{v}v\varepsilon'\mid (t_{v},v)\in \Omega\}$ is a subset of $\CN$.

\begin{lm}
The set of double cosets $B\bks G/H$ is in one-to-one correspondence with the set $\Omega$. Moreover, $\{\gamma(t_{v},v)\mid (t_{v},v)\in \Omega\}$ is a set of double coset representatives.
\end{lm}
\begin{proof}
Let $x=tv\varepsilon'$. Since the $(i,i)$-th entry of $\theta(x)x$ is $-1$ for $i\in \ScS^{(4)}_{w}$, $\theta(x)\neq x^{-1}$ and $x$ is not in $\CN$ if $\ScS^{(4)}\neq \emptyset$.
Thus  for $x\in \CN$  we may assume $x=ta_{v}b_{v}\varepsilon'$ for some $t=\diag\cpair{t_{1},t_{2},\dots,t_{2n},t^{-1}_{2n},\dots,t^{-1}_{2},t^{-1}_{1}}\in T$.
Let $t^{(i)}$ for $1\leq i\leq 3$ be elements in $T$, and $t^{(i)}_{j}$ be the $(j,j)$-th entry given by
$$
t^{(1)}_{j}=\begin{cases}
t_{j} & j\in \ScS^{(1)}_{v}\cup v(\ScS^{(1)}_{v})\\
1 & j\in\ScS_{2n}\smallsetminus (\ScS^{(1)}_{v}\cup v(\ScS^{(1)}_{v})),
\end{cases}
t^{(3)}_{j}=\begin{cases}
t_{j} & j\in \ScS^{(3)}_{v}\\
1 & j\in\ScS_{2n}\smallsetminus \ScS^{(3)}_{v},
\end{cases}
$$
and
$$
t^{(2)}_{j}=\begin{cases}
t_{j} & j\in \ScS^{(2)}_{v}\cup v(\ScS^{(2)}_{v})'\\
1 & j\in\ScS_{2n}\smallsetminus  (\ScS^{(2)}_{v}\cup v(\ScS^{(2)}_{v})').
\end{cases}
$$
Then  $t=t^{(1)}t^{(2)}t^{(3)}$. Let $d_{\tau}=\diag\{\tau I_{2n}, \tau^{-1}I_{2n}\}$. We also have the decomposition $d^{(1)}_{\tau}d^{(2)}_{\tau}d^{(3)}_{\tau}$ corresponding to $t^{(i)}$. Further,
\begin{align*}
&tvtv^{-1}=ta_{v}ta^{-1}_{v}b_{v}tb^{-1}_{v}=a_{v}t^{(1)}a^{-1}_{v}b_{v}t^{(2)}b^{-1}_{v}(t^{(3)})^{2},
\end{align*}
and
\begin{align*}
&\diag\cpair{\tau I_{2n},I_{2n}}w\diag\cpair{I_{2n},\tau^{-1} I_{2n}}w
=d^{(1)}_{\tau}a^{2}_{w}b^{2}_{w}d^{(3)}_{\tau}=d^{(1)}_{\tau}d^{(3)}_{\tau}.
\end{align*}

By $\theta(x)=x^{-1}$ and Equation~\eqref{eq:x-N}, we have
\begin{align*}
I_{2n}=&tvtv^{-1}\diag\cpair{\tau I_{2n},I_{2n}}v\diag\cpair{I_{2n},\tau^{-1} I_{2n}}v\\
=&d^{(1)}_{\tau}t^{(1)}a_{v}t^{(1)}a^{-1}_{v}\cdot t^{(2)}b_{v}t^{(2)}b^{-1}_{v}\cdot(t^{(3)})^{2}d^{(3)}_{\tau}.
\end{align*}
Thus
\begin{equation}\label{eq:t}
t_{i}=\begin{cases}
(\tau t_{v(i)})^{-1} & i\in \ScS^{(1)}_{v}\\
t_{v(i)'} &i\in \ScS^{(2)}_{v}\\
\pm1/\sqrt{\tau} & i\in \ScS^{(3)}_{v}.
\end{cases}
\end{equation}

Denote by $m^{+}_{v}=\#\{i\in \ScS^{(3)}_{v}\mid t_{i}=1/\sqrt{\tau}\}$ and $m^{-}_{v}=\#\{i\in \ScS^{(3)}_{v}\mid t_{i}=-1/\sqrt{\tau}\}$.
Since $x\in \CN$, there exists $g\in G$ such that $g\theta(g)^{-1}=x$ and then $g\varepsilon g^{-1}=tv\varepsilon'\varepsilon$.  The eigenvalues of $\varepsilon$ are $\pm \sqrt{\tau}$ and the dimension of the eigenspaces is $2n$. We have
$$
tv\varepsilon'\varepsilon=t^{(1)}a_{v}t^{(2)}b_{v}t^{(3)}\diag\cpair{\tau I_{2n}, I_{2n}},
$$
and decompose $\diag\cpair{\tau I_{2n}, I_{2n}}$ as $(\varepsilon'\varepsilon)^{(1)}(\varepsilon'\varepsilon)^{(2)}(\varepsilon'\varepsilon)^{(3)}$ similarly to the decomposition $t=t^{(1)}t^{(2)}t^{(3)}$.

By Equation~\eqref{eq:t},
the dimensions of the $\pm\sqrt{\tau}$-eigenspaces corresponding to the matrix $t^{(1)}a_{v}(\varepsilon'\varepsilon)^{(1)}$ are $2\#\ScS^{(1)}_{v}$, the dimensions of the $\pm\sqrt{\tau}$-eigenspaces corresponding to the matrix $t^{(2)}b_{v}(\varepsilon'\varepsilon)^{(2)}$ are $2\#\ScS^{(2)}_{v}$, and the dimensions of the $\pm\sqrt{\tau}$-eigenspaces for the matrix $t^{(3)}(\varepsilon'\varepsilon)^{(3)}$ are $2m^{+}$ and $2m^{-}$. Thus $m^{+}=m^{-}=\frac{1}{2}\ScS^{(3)}_{v}$.

Let $h\in T$. The $\theta$-twisted action $h*x$ is given by
\begin{equation}\label{eq:h}
hx\theta(h)^{-1}=h^{(1)}t^{(1)}(a_{v}(h^{(1)})^{-1}a^{-1}_{v})(h^{(2)})^{-1}t^{(2)}(c_{v}h^{(2)}c^{-1}_{v})v\varepsilon'.
\end{equation}
Choose the element $h\in T$ defined by
$$
h_{i}=\begin{cases}
t^{-1}_{i} &i\in \ScS^{(1)}_{v}\cup \ScS^{(2)}_{v}\\
1 & i\in \ScS_{2n}\smallsetminus (\ScS^{(1)}_{v}\cup \ScS^{(2)}_{v}).
\end{cases}
$$
By Equation~\eqref{eq:t} and \eqref{eq:h}, $h*x=t_{v}v\varepsilon'$.  Each $T$-orbit on $\CN$ has a nontrivial intersection with $\{t_{v}v\varepsilon'\mid (t_{v},v)\in \Omega\}$. In addition, by the above discussion, for the different pairs $(t_{v_{1}},v_{1})$ and $(t_{v_{2}},v_{2})$, the elements $t_{v_{1}}v_{1}\varepsilon'$ and $t_{v_{2}}v_{2}\varepsilon'$ are in different $T$-orbits.
This completes the proof of the lemma.
\end{proof}

\begin{cor}
The set $\cpair{\gamma(t_{v},v) \mid (t_{v},v)\in \Omega}$ is a set of coset representatives for $T\bks \Theta /H$.
Furthermore, $\Omega_{w}:=\cpair{g_{w}\gamma(t_{v},v) \mid (t_{v},v)\in \Omega}$ is a set of coset representatives for $T_{w}\bks \Theta_{T_{w}}/H$.
\end{cor}

Since $T_{w}$ and $H$ are connected, and $\theta$ and $F$ commute, the action of $F$ on $G$ induces an action on  the double cosets $T_{w}\bks \Theta_{T_{w}}/H$.
In addition, since the group $T_{w}\cap \gamma H \gamma^{-1}$ are connected for all $\gamma$ in $T_{w}\bks \Theta_{T_{w}}/H$, the action of $F$ on $T_{w}\bks \Theta_{T_{w}}/H$ is a permutation on the representatives $\Omega_{w}$.
Let $\Omega^{F}_{w}$ be the subset consisting of all fixed orbits under the action of $F$. Then 
 \begin{equation}
 T^{F}_{w}\bks \Theta^{F}_{T_{w}}/H^{F}\leftrightarrow \Omega^{F}_{w}.
 \end{equation}

%

Define
$$
\Omega^{w}_{\pm}=\cpair{(t_{v},v)\in \Omega\mid v\in Z_{W}(w) \text{ and } w(t_{v})_{i}w^{-1}=-(t_{v})_{i} \text{ for  } i\in \ScS^{(3)}_{v}}.
$$
\begin{lm}
If $T_{w}$ is an $F$-stable maximal torus corresponding to $w$, then $\Omega^{F}_{w}=\Omega^{w}_{\pm}$, i.e. there is a bijection
$$
T^{F}_{w}\bks \Theta^{F}_{T_{w}}/ H^{F}\leftrightarrow \Omega^{w}_{\pm}.
$$
\end{lm}
\begin{proof}
$F(g_{w}\gamma(t_{v},v))$ is  in the double coset $T_{w}\cdot g_{w}\gamma(t_{v'},v')\cdot H$ if and only if for some $t\in T$
\begin{align*}
&wF(t_{v})v\varepsilon'\varepsilon w^{-1}=t t_{v}v\varepsilon'\varepsilon t^{-1}\\
\iff &wF(t_{v})\diag\cpair{\tau I_{2n},I_{2n}}^{v}w^{-1} \cdot wvw^{-1}=t_{v}\diag\cpair{\tau I_{2n},I_{2n}}^{v}\cdot tvt^{-1}
\end{align*}
where $\diag\cpair{\tau I_{2n},I_{2n}}^{v}=v\cdot \diag\cpair{\tau I_{2n},I_{2n}}\cdot v^{-1}$.
Then $v$ commutes with $w$ in $W$.
Since $v^{2}=I_{4n}$ and $v\in Z_{W}(w)$, $wvw^{-1}$ is $T^{F}$-conjugate to $v$ and $tvt^{-1}\cdot wv^{-1}w^{-1}=t_{0}$ for some $t_{0}$ in the connected component of the $v$-split torus $\{t\in T\mid vtv^{-1}=t^{-1}\}$. The statement is equivalent to 
\begin{align*}
wF(t_{v})\diag\cpair{\tau I_{2n},I_{2n}}^{v}w^{-1}  =t_{v}\diag\cpair{\tau I_{2n},I_{2n}}^{v}\cdot t_{0}
\end{align*}
for some $t_{0}$ in the connected component of $v$-split torus.
Let $g=t_{v}\diag\cpair{\tau I_{2n},I_{2n}}^{v}$.
It is enough to prove that $wF(g)w^{-1}g^{-1}$ is in the connected component of $v$-split torus. 
Since $vgv^{-1}g=\tau I_{4n}$, the existence of $t_{0}$ is equivalent to $w(t_{v})_{i}w^{-1}=-(t_{v})_{i}$ for each $i\in \ScS^{(3)}_{v}$ under the assumption $v\in Z_{W}(w)$.  Then this lemma follows.
\end{proof}

According to \cite[Theorem 3.5.6]{Car93}, for a connected reductive group $G$, if its derived group is simply connected,  then $Z_{G}(s)$ is connected for any semi-simple element $s$. Since $G^{F}$ is the symplectic group $\Sp_{4n}(\BF_{q})$, $Z_{G}(t)$ and $T\cap fHf^{-1}$ in the definition \eqref{eq:eps} are connected. Thus $\epsilon_{T,f}$ is trivial.
If the character $\lam$ is trivial, then $T^F\bks \Theta^F_{T,\lam}/H^F$ is the same as $T^F\bks \Theta^F_{T}/H^F$.

\begin{ex}
Let $G$ be $\Sp_{4}$. By the lemmas above, it is easy to check that $\Omega$ contains 4 elements. 
If $(t_{v},v)$ is in $\Omega$, then $v$ is the identity $e$ or $(1,2)(1',2')$ or $(1,2')(1',2)$.
 If $v=e$, there are two choices of $t_{v}$, which are $\pm\diag\{1/\sqrt{\tau},-1/\sqrt{\tau},-\sqrt{\tau},\sqrt{\tau}\}$. 
 If $v\ne e$, $t_{v}$ is trivial. Hence we simply use $\{(1,2),(1,2'),(+,-),(-,+)\}$ to parametrize these 4 double cosets.
We give the centralizer  $Z_G(T_{w}\cap fH{f}^{-1})$ for each $F$-stable torus $T_{w}$ and the associated double cosets.

\begin{table}[h]
\begin{tabular}{|>{$}c<{$}| >{$}c<{$}| >{$}c<{$}| >{$}c<{$}| >{$}c<{$}| >{$}c<{$}|}
\hline
&(1^{2};0)&(2;0)&(1;1)&(0;2)&(0;1^{2})\\
\hline
(T_{w})^{F}&\BF^{\times}_{q}\times\BF^{\times}_{q}&
\BF^{\times}_{q^{2}}&\BF^{\times}_{q}\times\BF^{1}_{q^{2}}&\BF^{1}_{q^{4}}&\BF^{1}_{q^{2}}\times\BF^{1}_{q^{2}}\\
(1,2)&
\GL_{2}(\BF_{q})&\GL_{2}(\BF_{q})
&-&-&\RU_{2}(J,\BF_{q^{2}})\\
(+,-)&
-&\BF^{\times}_{q^{2}}&-&\BF^{1}_{q^{4}}&-\\
(-,+)&
-&\BF^{\times}_{q^{2}}&-&\BF^{1}_{q^{4}}&-\\
(1,2')&
\GL_{2}(\BF_{q})&\RU_{2}(\BF_{q^{2}})
&-&-&\RU_{2}(J,\BF_{q^{2}})\\
R^{\triv}_{w}&2&2&0&2&-2\\
\hline
\end{tabular}
\end{table}
 The first row lists the pairs of partitions in $\CP_{2}$ parametrizing  the $G^{F}$-conjugacy classes of $F$-stable maximal tori.
The last row is the dimension formula $<R^{\triv}_{w},\Ind^{G^{F}}_{H^{F}}\triv>$.
The symbol `$-$' means that the corresponding double coset is not $F$-stable and not in the set $\Omega^{w}_{\pm}$.
\end{ex}

\subsection{Decomposition of $\Xi_{n}$}
There is a one-to-one correspondence between pairs of partitions $(\alpha;\beta)$ such that $|\alpha|+|\beta|=n$ and the conjugacy classes in $W_{n}$.
If $i$ and $j$ are parts of $\alpha$ and $\beta$ respectively, we map $i$ and $j$ to Coxter elements in $S_{i}$ and $W_{j}$ respectively, which extends to a bijection between pairs of partitions $(\alpha;\beta)$ in $\CP_{n}$ and the conjugacy classes in $W_{n}$.
Denote by $Cl_{W_{|\alpha|+|\beta|}}(\alpha;\beta)$ the conjugacy class corresponding to $(\alpha; \beta)$.
If $w$ is in $W_{2n}$ of cycle-type $(\alpha;\beta)$, then
$$
T^{F}_{w}=\prod_{i}\BF_{q^{i}}^{\times m_{i}(\alpha)}\times \prod_{i} (\BF_{q^{2i}}^{1})^{m_{i}(\beta)}
\text{ and }
\text{$\BF_{q}$-rank}(T_{w})=\ell(\alpha),
$$
where $\BF^{1}_{q^{2i}}$ is the unitary group $\RU_{1}(\BF_{q^{2i}})$.

Let $(\alpha;\beta)$ be in $\CP_{2n}$.
Let $I_{1}(\alpha)=\{(1\ 1'),(2\ 2'),\dots,(\ell(\alpha)\ \ell(\alpha)'))\}$ where $(i\ i')$ is a permutation in $W_{2n}$,
and  $I_{2}(\beta)=\{1,2,\dots,\ell(\beta)\}$  be sets indexing the parts of $\alpha$ and $\beta$ respectively. Define
$$\RP_{1}(\alpha)=\{((i\ i'),j)\mid (i\ i')\in I_{1}(\alpha)\text{ and } j\in \BZ_{\alpha_{i}} \}$$
and
$$
\RP_{2}(\beta)=\{(i,j)\mid i\in I_{2}(\beta)\text{ and }j\in\BZ_{2\beta_{i}}\}.
$$
Then $w$ is a permutation of $\RP_{1}(\alpha)$ and $\RP_{2}(\beta)$ by $w(i,j)=(i,j+1)$.

For each $(t_{v},v)\in \Omega^{w}_{\pm}$, $v$ induces permutations of $\RP_{1}(\alpha)$ and $\RP_{2}(\beta)$ of order 2, denoted by $\bar{v}$ and $\tilde{v}$ respectively. There are unique elements $j_{1}(v,i)$ and $j_{2}(v,i)$  such that $v(i,j)=(\bar{v}(i),j+j_{1}(v,i))$ for $(i,j)\in \RP_{1}(\alpha)$ and $v(i,j)=(\tilde{v}(i),j+j_{2}(v,i))$ for $(i,j)\in \RP_{2}(\beta)$.
Note that for $(i\ i')$ and $(j\ j')$ in $I_{1}(\alpha)$, $\bar{v}(i\ i')=(j\ j')$ and $\bar{v}(i\ i')=(j'\ j)$ are considered as two different permutations of $\{(i\ i'),(j\ j')\}$.

 Define
$$
\bar{\iota}_{v}(\alpha)=\#\{(i\ i')\in I_{1}(\alpha)\mid \bar{v}(i\ i')=(i'\ i) \}
$$
and
$$
\bar{\iota}^{f}_{v}(\alpha)=\#\{(i\ i')\in I_{1}(\alpha)\mid \bar{v}(i\ i')=(i\ i') \text{ and } j_{1}(v,i)=0\}.
$$
 Also define
\begin{align*}
\tilde{\iota}_{v}(\beta)=\#\{i\in I_{2}(\beta)\mid \tilde{v}(i)=j \text{ and } i\ne j\}.
\end{align*}
Note that $\tilde{\iota}_{v}(\beta)$ is always even and
 $\Omega^{w}_{\pm}\ne \emptyset$ if and only if $m_{k}(\alpha)$ and $m_{k}(\beta)$ are even for  each odd part  $k$ of $\alpha$ and $\beta$.

Define a $W_{2n}$-module $\Xi_{n}$ such that the character of $\Xi_{n}$ is given by
\begin{equation}\label{Xi}
\tr(w,\Xi_{n})=\apair{\tr(\cdot , R^{\triv}_{T_{w}}),\Ind^{G^{F}}_{H^{F}} \mathds{1}}.
\end{equation}

\begin{lm}
\begin{align*}
&\tr(w,\Xi_{n})=\sum_{(t_{v}, v)\in \Omega^{w}_{\pm}}(-1)^{\ell(\alpha)+\bar{\iota}_{v}(\alpha)+\bar{\iota}_{v}^{f}(\alpha)+\frac{\tilde{\iota}_{v}(\beta)}{2}}
\end{align*}
where $(\alpha;\beta)$ is the pair of partitions associated to $w$.
\end{lm}
\begin{proof}
We decompose the symplectic space $V=(\BF_{q})^{4n}$ according to the invariant subspaces of $T^{F}$  
$$
V=\oplus_{(i,i')\in I_{1}(\alpha)}(V_{i}\oplus V_{i'})\oplus \oplus_{i\in I_{2}(\beta)}V_{i},
$$
where $T^{F}\vert_{V_{i}\oplus V_{i'}}\cong \BF^{\times}_{q^{\alpha_{i}}}$ for $(i\ i')\in I_{1}(\alpha)$ and $T^{F}\vert_{V_{i}}\cong \BF^{1}_{q^{2\beta_{i}}}$ for $i\in I_{2}(\beta)$.
Fix an element $(t_{v},v)\in \Omega^{w}_{\pm}$ and  denote by $\gamma$ the corresponding representative of the double cosets $T^{F}_{w}\bks\Theta^{F}_{T_{w}}/H^{F}$.

Let $(i\ i')$ be an index in $I_{1}(\alpha)$ and assume $\bar{v}(i\ i')=(j\ j')$ or $(j'\ j)$. If $i\ne j$, then the restriction of $Z_{G}(T_{w}\cap \gamma H\gamma^{-1})$ into the invariant  space $V_{i}\oplus V_{j}$ or $V_{i}\oplus V_{j'}$ is isomorphic to $\GL_{2}(V_{i}\oplus V_{j})$ and its $\BF_{q}$-rank is 2.
If $i=j$, then $\alpha_{i}$ is even due to the existence of $t_{v}$ and we have two cases, $j_{1}(v,i)\ne 0$ and  $j_{1}(v,i)=0$.  If $\bar{v}(i\ i')=(i\ i')$ and $j_{1}(v,i)=0$, then the restriction of $Z_{G}(T_{w}\cap \gamma H\gamma^{-1})$ into the invariant space $V_{i}\oplus V_{i'}$ is the same as $T^{F}\vert_{V_{i}\oplus V_{i'}}\cong \BF^{\times}_{q^{\alpha_{i}}}$ and its $\BF_{q}$-rank is 1.
If $\bar{v}(i\ i')=(i\ i')$ and $j_{1}(v,i)\ne 0$, the restriction of $Z_{G}(T_{w}\cap \gamma H\gamma^{-1})$ is isomorphic to $\GL_{2}(V_{i})$ but invariant under $F^{\frac{\alpha_{i}}{2}}$ and its $\BF_{q}$-rank is 2. If $\bar{v}(i,i')=(i',i)$,  the restriction of $Z_{G}(T_{w}\cap \gamma H\gamma^{-1})$ is isomorphic to $\RU_{2}(V_{i}\oplus V_{i'})$ and its $\BF_{q}$-rank is 1.

Let $i$ be an index in $I_{2}(\beta)$ and $\bar{v}(i,i')=(j,j')$ or $(j',j)$. If $i\ne j$, the restriction of $Z_{G}(T_{w}\cap \gamma H\gamma^{-1})$ into the invariant space $V_{i}\oplus V_{j}\oplus V_{j'}\oplus V_{i'}$ is isomorphic to $\RU_{2}(V_{i}\oplus V_{j})$ and its $\BF_{q}$-rank is 1. If $i=j$, then $j_{2}(\beta,i)=0$ and $\beta_{i}$ is even due to the existence of $t_{v}$.
The restriction of $Z_{G}(T_{w}\cap \gamma H\gamma^{-1})$ is the same as $T^{F}\vert_{V_{i}}\cong \BF^{1}_{q^{2\beta_{i}}}$ and its $\BF_{q}$-rank is 0.

In sum, the $\BF_{q}$-rank of $Z_{G}(T_{w}\cap \gamma H\gamma^{-1})$ and $\bar{\iota}_{v}(\alpha)+\bar{\iota}^{f}_{v}(\alpha)+\frac{\tilde{\iota}_{v}(\beta)}{2}$ have the same parity. 
By Proposition~\ref{pro:Lusztig},
this lemma follows.
\end{proof}

Recall that $W_{2n}\cong S_{2n}\ltimes (\BZ_{2})^{2n}$ acts on $\ScT_{2n}$. Take the subgroup, isomorphic to $S_{n}\times S_{n}$, of $W_{2n}$ consisting of all permutations on the  $\{1,2,\dots,n\}\times\{n+1,\dots,2n\}$. 
Let $\sig_{n}=\prod^{n}_{i=1}(i,2n+1-i)$. Then $\sig_{n}$  in $S_{2n}$ normalizes $S_{n}\times S_{n}$, 
that is, $\sig_{n}$ is the longest Weyl element $w_{2n}$.
Let $K_{n}=(\langle\sig_{n}\rangle\ltimes (S_{n}\times S_{n}))\ltimes (\BZ_{2})^{2n}$ regarded as a subgroup of $W_{2n}$ and
$\sgn_{K}$ be a character of $K_{n}$ lifting from the non-trivial character of the group $\{e,\sig_{n}\}$.
Define a virtual module of $W_{2n}$
$$\kappa_{n}=\Ind^{W_{2n}}_{K_{n}}\triv-\Ind^{W_{2n}}_{K_{n}}\sgn_{K}.$$

Let $N_{n}$ be the centralizer $Z_{W_{2n}}(\sig_{n})$. Then $N_{n}$ is isomorphic to $(S_{2}\wr S_{n})\ltimes (\BZ_{2})^{n}$ where $S_{2}\wr S_{n}$ is the wreath product.
Define  a character $\sgn_{N}$ of $N_{n}$  by
$$
\sgn_{N}(h)=(-1)^{\#\cpair{h(1),h(2),\dots,h(n)}\cap\cpair{1',2',\dots,n'}}.
$$
Let  $\nu_{n}=\Ind^{W_{2n}}_{N_{n}}\sgn_{N}$.

A function $f$ on the set $\CP_{2n}$ is called {\it multiplicative} if $f(\alpha;\beta)=\prod_{i}f(i^{m_{i}(\alpha)};0)f(0;i^{m_{i}(\beta)})$ where $(i^{m_{i}})$ is a partition of $i\cdot m_{i}$.
Now, we decompose $\kappa_{n}$ and $\nu_{n}$ into irreducible representations of $W_{2n}$.
\begin{lm}\label{lm:N}
$$
\nu_{n}=\sum_{\alpha\vdash n}\binom{\alpha}{\alpha}.
$$
Moreover, the character function $\tr(\cdot, \nu_{n})$ is multiplicative and if $w$ is an element of cycle-type $(\alpha;\beta)$ then the following holds:
\begin{enumerate}
 \item $\tr(w,\nu_{n})=0$ when $w$ is of cycle-type $(i^{m};0)$ or $(0;i^{m})$, and $i$ and $m$ are odd;
 \item $\tr(w,\nu_{n})=i^{m}\frac{(2m)!}{m!}$ when $w$ is of cycle-type $(i^{2m};0)$ and $i$ is odd;
 \item $\tr(w,\nu_{n})=(-i)^{m}\frac{(2m)!}{m!}$ when $w$ is of cycle-type $(0;i^{2m})$ and $i$ is odd;
 \item $\tr(w,\nu_{n})=0$ when $w$ is of cycle-type $((2i)^{m};0)$ or $(0;(2i)^{m})$, and  $m$ is odd;
 \item $\tr(w,\nu_{n})=(2i)^{\frac{m}{2}}\frac{m!}{(m/2)!}$ when $w$ is of cycle-type $((2i)^{m};0)$ and $m$ is even;
 \item $\tr(w,\nu_{n})=(-2i)^{\frac{m}{2}}\frac{m!}{(m/2)!}$ when $w$ is of cycle-type $(0;(2i)^{m})$ and $m$ is even.
\end{enumerate}

\end{lm}
\begin{proof}
First, we give a set of representatives for the double cosets 
$(W_{n_{1}}\times W_{n_{2}})\bks W_{2n}/N_{n}$ where $n_{1}\leq n_{2}$ and $n_{1}+n_{2}=2n$, which is the same as $S_{n_{1}}\times S_{n_{2}}\bks S_{2n}/S_{2}\wr S_{n}$.
Let
\begin{equation}\label{eq:h}
h_{r,n_{1}}=\prod^{r}_{i=1}(n_{1}+1-i,2n+1-i)\text{ for } 0\leq r\leq [\frac{n}{2}].
\end{equation}
Then $\{h_{r,n_{1}}\mid 0\leq r\leq [\frac{n}{2}]\}$ is a set of representatives for the double cosets $S_{n_{1}}\times S_{n_{2}}\bks S_{2n}/S_{2}\wr S_{n}$. Denote by $\bar{h}_{r,n_{1}}$ the embedded element in $W_{2n}$ and $\{\bar{h}_{r,n_{1}}\mid 0\leq r\leq [\frac{n}{2}]\}$ is a set of representatives for the double cosets $W_{n_{1}}\times W_{n_{2}}\bks W_{2n}/N_{n}$.

Let $\binom{\alpha}{\beta}$ be the irreducible module of $W_{2n}$ given in Section~\ref{sec:weyl}. If $|\alpha|\ne |\beta|$, let $n_{1}=\min\{|\alpha|,|\beta|\}$.
Then $g=(n,n')(n+1,(n+1)')$ commutes with $\bar{h}_{r, n_{1}}$ for all $0\leq r\leq [\frac{n_{1}}{2}]$.
Since $g\in N_{n}$ and $g=\bar{h}_{r,n_{1}}g\bar{h}_{r,n_{1}}^{-1}\in W_{|\alpha|}\times W_{|\beta|}$, we have $\sgn_{N}(g)=-1$ and $\chi_{|\beta|}(g)=1$ if $|\alpha|<|\beta|$ or $\bar{\rho}_{S}(\alpha)(g)=1$ if $|\alpha|>|\beta|$.
By Mackey's theorem,
$$
\langle \binom{\alpha}{\beta},\nu_{n}\rangle_{W_{2n}}=0 \text{ when }
|\alpha|\ne |\beta|.
$$

Now, we only need  consider the case $|\alpha|=|\beta|=n$. If $r\ne 0$, then $\bar{h}_{r,n}g\bar{h}_{r,n}^{-1}=(2n+1,(2n+1)')(n+1,(n+1)')$ is in $W_{|\beta|}$. Also $\sgn_{N}(g)=-1$ and $\chi_{|\beta|}(\bar{h}_{r,n}g\bar{h}_{r,n}^{-1})=1$.
Denote by $N'_{n}=(W_{n}\times W_{n})\cap N_{n}$. This is  isomorphic to $S^{\triangle}_{n}\ltimes (\BZ_{2})^{n}$ where $S^{\triangle}_{n}$ is the subgroup of  $S_{n}\times S_{n}$, and commutes with $\sig_{n}$.
 Using Mackey's theorem and Frobenius
reciprocity, we have
$$
\langle \binom{\alpha}{\beta},\nu_{n}\rangle_{W_{2n}}=
\langle\rho_{S}(\alpha)\otimes \rho_{S}(\beta),\triv\rangle_{S^{\triangle}_{n}}
=\langle \rho_{S}(\alpha), \rho_{S}(\beta)\rangle_{S_{n}}.
$$

Next, let us calculate $\tr(w,\nu_{n})$ where $w$ is an element of cycle-type $(\alpha;\beta)$.
We discuss the values at $w$ in each case separately.
If $m_{i}(\alpha)$ or $m_{i}(\beta)$ is odd for some odd $i$, then $Cl_{W_{2n}}(w)\cap N_{n}=\emptyset$ and $\tr(w,\nu_{n})=0$.

If  $w$ is of cycle-type $(i^{2m};0)$ or $(0;i^{2m})$ and $i$ is odd, then $Cl_{W_{2n}}(w)\cap N_{n}$ contains at most one $N_{n}$-conjugacy class in $N_{n}$. If $w$ is in $N_{n}$, then
$$
\tr(w,\nu_{n})=\frac{|Z_{W_{2n}}(w)|}{|Z_{N_{n}}(w)|}=\frac{i^{2m}(2m)!2^{2m}}{i^{m}m!2^{m}2^{m}}=i^{m}\frac{(2m)!}{m!}
$$
when  $w$ is of cycle-type $(i^{2m};0)$ and
$$
\tr(w,\nu_{n})=(-1)^{m}\frac{|Z_{W_{2n}}(w)|}{|Z_{N_{n}}(w)|}=\frac{(2i)^{2m}(2m)!}{i^{m}m!2^{m}2^{m}}=(-i)^{m}\frac{(2m)!}{m!}
$$
when  $w$ is of cycle-type $(0;i^{2m})$.

Let $w$ be an element of cycle-type $(0;(2i)^{m})$. If $m$ is odd, then $Cl_{W_{2n}}(w)\cap N_{n}=\emptyset$. If $m$ is even, then $Cl_{W_{2n}}(w)\cap N_{n}$ contains at most one $N_{n}$-conjugacy class in $N_{n}$. If $w$ is in $N_{n}$, then
$$
\tr(w,\nu_{n})=(-1)^{\frac{m}{2}}\frac{Cl_{W_{2n}}(w)}{Cl_{N_{n}}(w)}=(-1)^{\frac{m}{2}}\frac{(4i)^{m}m!}{(2i)^{\frac{m}{2}}\frac{m}{2}!2^{\frac{m}{2}}\cdot 2^{\frac{m}{2}}}
=\frac{(-2i)^{\frac{m}{2}}m!}{(m/2)!}.
$$

Let $w$ be an element of cycle-type $((2i)^{m};0)$ and $Cl_{W_{2n}}(w)\cap N_{n}=\emptyset$. In this case, $Cl_{W_{2n}}(w)\cap N_{n}$ contains $2[\frac{m}{2}]+1$ conjugacy classes of $N_{n}$. By the formula for the induced character,
$$
\tr(w,\nu_{n})=\frac{Z_{W_{2n}}(w)}{Z_{N_{n}}(w)}=\frac{(2i)^{m}m!2^{m}}{(2i)^{\frac{m}{2}}\frac{m}{2}!2^{\frac{m}{2}}\cdot 2^{\frac{m}{2}}}
=(2i)^{\frac{m}{2}}\frac{m!}{(m/2)!}
$$
when $m$ is even and $\tr(w,\nu_{n})=0$ when $m$ is odd.
In addition, it is easy to check that $\tr(\cdot,\kappa_{n})$ is multiplicative. This completes the proof.
\end{proof}

\begin{lm}\label{lm:K}
$$
\kappa_{n}=\sum^{n}_{i=0}(-1)^{i}\binom{i,2n-i}{0}.
$$
Moreover, the character $\tr(\cdot, \kappa_{n})$ is multiplicative and if $w$ is an element of cycle-type $(\alpha;\beta)$ then
\begin{equation}\label{eq:K-tr}
\tr(w,\kappa_{n})=\begin{cases}
0 &(\alpha;\beta)=((2i+1)^{m};0)\text{ or }(0;(2i+1)^{m})\\
2^{m}  &(\alpha;\beta)=((2i)^{m};0)\text{ or } (0;(2i)^{m}).
\end{cases}
\end{equation}

\end{lm}
\begin{proof}
Since the characters $\triv$ and $\sgn_{K}$ of $K_{n}$ are trivial on the subgroup $(\BZ_{2})^{2n}$ of $W_{2n}$,
it is sufficient to show that
\begin{equation}\label{eq:K-S}
\Ind^{S_{2n}}_{K'_{n}}\triv-\Ind^{S_{2n}}_{K'_{n}}\sgn_{K'}=\sum^{n}_{i=0}(-1)^{i}\rho_{S}(i,2n-i),
\end{equation}
where $K'_{n}=\langle\sig_{n}\rangle\ltimes (S_{n}\times S_{n})$ is the subgroup  $K_{n}\cap S_{2n}$ and $\sgn_{K'}$ is the restriction of $\sgn_{K}$ on $K'_{n}$.

First, applying Pieri's formula, we have
\begin{equation}\label{eq:K-r}
\Ind^{S_{2n}}_{S_{n}\times S_{n}}\triv=\Ind^{S_{2n}}_{K'_{n}}\triv+\Ind^{S_{2n}}_{K'_{n}}\sgn_{K'}=\sum^{n}_{i=0}\rho_{S}(i,2n-i).
\end{equation}
Let $g_{r}=\prod^{r}_{i=1}(i,2n+1-i)$ be an element in $S_{2n}$ ($g_{r}=e$ when $r=0$, and $g_{n}=\sig_{n}$). Then $g_{r}$ for $0\leq r\leq n$ are a complete set of representatives for the double cosets $(S_{n}\times S_{n})\bks S_{2n}/ (S_{n}\times S_{n})$ and $S_{2n}=\coprod^{[n/2]}_{r=0}K'_{n}g_{r}K'_{n}$. Since $g_{r}$ and $\sig_{n}$ commute for all $r$, $\sig_{n}$ is in $K'_{n}\cap g_{r} K'_{n}g^{-1}_{r}$ for each double coset $K'_{n}g_{r}K'_{n}$ in $S_{2n}$. Since $\triv(\sig_{n})\ne \sgn_{K'}(\sig_{n})$, by Mackey's  theorem,
\begin{equation}\label{eq:K-triv-sgn}
\apair{\Ind^{S_{2n}}_{K'_{n}}\triv,\Ind^{S_{2n}}_{K'_{n}}\sgn_{K'}}_{S_{2n}}=0
\end{equation}
and
$$
\apair{\Ind^{S_{2n}}_{K'_{n}}\triv,\Ind^{S_{2n}}_{K'_{n}}\triv}_{S_{2n}}=\#K'_{n}\bks S_{2n}/K'_{n}=[\frac{n}{2}]+1.
$$

Referring to Macdonald \cite[VII (2.4)]{Mac95},
$$
\apair{\Ind^{S_{2n}}_{S_{n}\times S_{n}}\triv,\Ind^{S_{2n}}_{S_{2}\wr S_{n}}\triv}_{S_{2n}}=[\frac{n}{2}]+1.
$$
Also we have
$$
\apair{\Ind^{S_{2n}}_{K'_{n}}\triv,\Ind^{S_{2n}}_{S_{2}\wr S_{n}}\triv}_{S_{2n}}=\#K'_{n}\bks S_{2n}/S_{2}\wr S_{n}=[\frac{n}{2}]+1.
$$
Hence $\Ind^{S_{2n}}_{K'_{n}}\triv$ is a submodule of $\Ind^{S_{2n}}_{S_{2}\wr S_{n}}\triv$. By~\eqref{eq:K-r},
$$
\Ind^{S_{2n}}_{K'_{n}}\triv=\sum^{[n/2]}_{i=0}\rho_{S}(2i,2n-2i).
$$
Then Equation~\eqref{eq:K-S} follows by~\eqref{eq:K-triv-sgn}.

Next, let us calculate $\tr(w,\kappa_{n})$ where $w$ is an element of cycle-type $(\alpha;\beta)$.
 If a part $i$ of $\alpha$ or $\beta$ is odd, then $Cl_{G}(w)\cap K_{n}=Cl_{K_{n}}(w)$ is in the subgroup $K'_{n}\ltimes (\BZ_{2})^{2n}$. By the formula for the induced character, $\tr(w,\kappa_{n})=0$.

If $Cl_{W_{2r}}(w)$ and the left coset $\sig_{n}\cdot K'_{n}\ltimes (\BZ_{2})^{2n} $ have a non-empty intersection, then $Cl_{G}(w)\cap K_{n}$ is the disjoint union of at most two $K_{n}$-conjugacy classes.
If there are two conjugacy classes, one conjugacy class is in $K'_{n}\ltimes (\BZ_{2})^{2n}$ and  the other is in $\sig_{n}\cdot K'_{n}\ltimes (\BZ_{2})^{2n}$. In all cases, we have
$$
\tr(w,\kappa_{n})=\frac{2|W_{2n}|}{|K_{n}|}\cdot \frac{|Cl_{W_{2n}}(w)\cap \sig_{n}\cdot K'_{n}\ltimes (\BZ_{2})^{2n}|}{|Cl_{W_{2n}}(w)|}
=\frac{2Z_{W_{2n}}(w)}{Z_{K_{n}}(w)}.
$$
Hence, $\tr(w,\kappa_{n})$ is multiplicative and
Equation~\eqref{eq:K-tr} follows since 
$$Z_{W_{2n}}(w)=(2i)^{m}m!2^{m}\text{ and } Z_{K_{n}}(w)=2\cdot i^{m}m!2^{m}$$ when $(\alpha;\beta)=((2i)^{m};0)$ and
$$Z_{W_{2n}}(w)=(4i)^{m}m! \text{ and } Z_{K_{n}}(w)=2\cdot (2i)^{m}m!$$
 when $(\alpha;\beta)=(0;(2i)^{m})$.
This completes the proof.
\end{proof}

\begin{pro}\label{pro:Xi-partitions}
\begin{equation}\label{eq:Xi}
\Xi_{n}=\sum_{r=0}^{n}\sum_{\beta\vdash n-r}\sum^{r}_{i=1}
(-1)^{i}\binom{(i,2r-i)\cdot \beta}{\beta}
\end{equation}
where $(1^{i})\cdot (1^{2r-i})\cdot \beta$ corresponds to the representation
$$\Ind^{S_{n+r}}_{S_{2r}\times S_{n-r}}\rho_{S}(i,2r-i)\otimes \rho_{S}(\beta).$$
\end{pro}
\begin{proof}
It is equivalent to show the following identity
\begin{equation}\label{eq:Xi-k-n}
\Xi_{n}=\sum^{n}_{r=0}\Ind^{W_{2n}}_{W_{2r}\times W_{2(n-r)}}\kappa_{r}\otimes \nu_{n-r}.
\end{equation}
We will match the characters of the two sides.

Let $w$ be an element in $W_{2n}$ of cycle-type $(\alpha;\beta)$.
The set $\Omega^{w}_{\pm}$ may be written as a disjoint union $\coprod^{n}_{r=0}\Omega^{w}_{\pm,r}$, where
$$
\Omega^{w}_{\pm,r}=\Omega^{w}_{\pm}\cap\{v\in W_{2n}\mid \ScS^{(3)}_{v}=2r\}.
$$
Define the set
$$
\CP_{r}(\alpha;\beta)=\{(a,b)\in\CP_{2r}\times\CP_{2n-2r}\mid
a\cup b=(\alpha;\beta)\}
$$
where the union $a\cup b$ is the component-wise union.
We may further decompose $\Omega^{w}_{\pm,r}$ as
$$
\Omega^{w}_{\pm,r}=\coprod_{(a,b)\in \CP_{r}(\alpha;\beta)}\Omega^{w'}_{\pm,r}\times\Omega^{w''}_{\pm,0}.
$$
where $w'$ and $w''$ are of cycle-types $a$ and $b$ in $W_{2r}$ and $W_{2n-2r}$ respectively.
Note that if $(t_{v'},v')$ is in $\Omega^{w'}_{\pm,r}$, then $v'$ is the identity and we take $t'\in \Omega^{w'}_{\pm,r}$;
if $(t_{v''},v'')$ is in $\Omega^{w''}_{\pm,0}$ then $t_{v''}=1$ and we take $v''\in \Omega^{w''}_{\pm,0}$.
It is easy to observe that
$\ell(\alpha)=\ell(\alpha')+\ell(\alpha'')$,
$\bar{\iota}^{f}_{v}(\alpha)=\ell(\alpha')$,
$\bar{\iota}_{v}(\alpha)=\bar{\iota}_{v''}(\alpha'')$, and
$\tilde{\iota}_{v}(\beta)=\ell(\beta'')$ where $a=(\alpha';\beta')$ and $b=(\alpha'';\beta'')$.
Then
\begin{align*}
\tr(w,\Xi_{n})=&\sum^{n}_{r=0}\sum_{(t_{v},v)\in \Omega^{w}_{\pm,r}}\sum_{(a,b)\in \CP_{r}(\alpha;\beta)}\\
&(\sum_{t'\in\Omega^{w'}_{\pm,r}}(-1)^{2\ell(\alpha')})(\sum_{v''\in \Omega^{w''}_{\pm,0}}(-1)^{\ell(\alpha'')+\bar{\iota}_{v''}(\alpha'')+\ell(\beta'')/2}).
\end{align*}
We consider the characters $\tr(w',\Xi'_{r})$ and $\tr(w'',\Xi''_{n-r})$ on $W_{2r}$ and $W_{2n-2r}$ defined by
$\tr(w',\Xi'_{r}):=\#\Omega^{w'}_{\pm,r}$ and
$$
\tr(w'',\Xi''_{n-r}):=\sum_{v''\in \Omega^{w''}_{\pm,0}}(-1)^{\ell(\alpha'')+\bar{\iota}_{v''}(\alpha'')+\ell(\beta'')/2}.
$$
By the definition, the class functions $\tr(w',\Xi'_{r})$ and $\tr(w'',\Xi''_{n-r})$ are multiplicative.
In order to prove~\eqref{eq:Xi-k-n}, it is enough to  show that  the characters $\tr(w',\Xi'_{r})$ and $\tr(w'',\Xi''_{n-r})$ match $\tr(\cdot,\kappa_{2r})$ and $\tr(\cdot,\nu_{2n-2r})$ on the conjugacy classes of cycle-types $(i^{m};0)$ and $(0;i^{m})$. This is easily verified by Lemma~\ref{lm:K} and Lemma~\ref{lm:N} and the definition of $\tr(\cdot,\Xi'_{r})$ and $\tr(\cdot,\Xi''_{n-r})$. Then this lemma follows.
\end{proof}

\section{Distinguished Symbols}\label{sec:main}
In this last section, we will use the Littlewood-Richardson  rule to decompose $\Xi_{n}$ into irreducible representations of $W_{2n}$ and then divide those constituents into a sum of virtual cells. Doing so, we conclude that there is a bijection between symbols in those cells and the distinguished unipotent representations.

The set-theoretic difference of the Young diagrams of a pair of partitions $(\alpha;\beta)$ is called a skew diagram of shape $\alpha/\beta$ and size $|\alpha/\beta|$ is $|\alpha|-|\beta|$.
Let $(\alpha/\beta)_{i}=\alpha_{i}-\beta_{i}$.
A skew diagram is a \emph{horizontal strip} (resp.~\emph{vertical strip}) if  it contains at most one box in each column (resp.~row).
For a skew diagram $\eta$, let $h(\eta)$ be the horizontal strip obtained by removing all columns from $\eta$ which contain  more than one box  and $v(\eta)$ be the skew diagram  obtained by removing $h(\eta)$ from $\eta$. For example,
$$
\eta=\ydiagram{2+3,2+2,2,1} \qquad
h(\eta)=\ydiagram{4+1,0,1+1,0} \qquad
v(\eta)=\ydiagram{2+2,2+2,1,1}.
$$
Denote by $|h(\eta)|$ (resp. $|v(\eta)|$) the number of boxes of $h(\eta)$ (resp. $v(\eta)$).

A horizontal strip $\eta$ is called {\it even} if the number of boxes in each row of $\eta$ is even.
Let $\Gamma_{2}$ be the set of all skew diagrams which contain at most 2 boxes in each column and is of even size, and $\Gamma^{\circ}_{2}$ be a subset of $\Gamma_{2}$ of all skew diagrams $\eta$ such that $h(\eta)$ is even.

Let $\eta$ be in $\Gamma_{2}$. Define a {\it tableau} $T$ of shape $\eta$ by putting the integers `1' or `2' in each box of $\eta$ such that the follow Littlewood-Richardson condition holds:
The integers in all boxes are listed from right to left then from top to bottom, then at the first $k$ entries in this list  for each $1\leq k\leq |\eta|$, each `1' occurs at least as many times as `2'.

Define $\Tab(\eta)$ to be the set of all tableaus of shape $\eta$ and $\Tab(\eta)_{i}$ to be the subset of $\Tab(\eta)$ consisting of tableaus of $i$ `2's.
 For example, when $i>\frac{|\eta|}{2}$, $\Tab(\eta)_{i}=\emptyset$ and $|\Tab(\eta)_{i}|=0$.
By convention,  if $|\eta|=0$, define $\Tab(\eta)=\{\emptyset\}$ and $|\Tab(\eta)|=1$, where $|\Tab(\eta)|$ (resp. $|\Tab(\eta)_{i}|$) is the size of the set $\Tab(\eta)$ (resp. $\Tab(\eta)_{i}$).
Also we have the disjoint union $\Tab(\eta)=\cup_{i\geq 0}\Tab(\eta)_{i}$.

Next, we use the Littlewood-Richardson rule to decompose $\Xi_{n}$ into irreducibles and get the following lemma.
\begin{lm}\label{lm:Xi-even}
$$
\Xi_{n}=\sum_{\substack{(\alpha;\beta)\in\CP_{2n}\\ \alpha/\beta\in \Gamma^{\circ}_{2}}}(-1)^{\frac{|v(\alpha/\beta)|}{2}}\binom{\alpha}{\beta}.
$$
\end{lm}
\begin{proof}
By Equation~\eqref{eq:Xi}, if $\langle \Xi_{n},\binom{\alpha}{\beta}\rangle\ne 0$, then $\beta\leq \alpha$ and $\alpha/\beta\in \Gamma_{2}$.
If $\alpha=\beta$, then $\alpha/\beta$ is in $\Gamma^{\circ}_{2}$ and  $\langle \Xi_{n},\binom{\alpha}{\beta}\rangle=1$.

If $\beta<\alpha$, by the Littlewood-Richardson rule,
\begin{align*}
\langle \Xi_{n},\binom{\alpha}{\beta}\rangle=&\sum^{|\alpha/\beta|/2}_{i=0}(-1)^{i}|\Tab(\alpha/\beta)_{i}|=\sum^{|\alpha/\beta|/2}_{i=|v(\alpha/\beta)|/2}(-1)^{i}|\Tab(\alpha/\beta)_{i}|.
\end{align*}
Since $
|\Tab(\alpha/\beta)_{i}|=|\Tab(h(\alpha/\beta))_{i-|v(\alpha/\beta)|/2}|,
$
we have
$$
\langle \Xi_{n},\binom{\alpha}{\beta}\rangle=(-1)^{|v(\alpha/\beta)|/2}\sum^{|h(\alpha/\beta)|/2}_{i=0}(-1)^{i}|\Tab(h(\alpha/\beta))_{i}|.
$$

For a tableau $T$ in $\Tab(\eta)$, let $|T_{2}|$ be the number of boxes filled with 2's.
For a subset $\CT$ of $\Tab(\eta)$, define
\begin{equation}
\zeta(\CT)=\sum_{T\in \CT}(-1)^{|T_{2}|}.
\end{equation}
Next, we will show that for a horizontal strip $\eta$ of even size
\begin{equation}\label{eq:zeta}
\zeta(\Tab(\eta))=\begin{cases}
1 & \text{ if $\eta$ is even;}\\
0 & \text{ otherwise.}
\end{cases}
\end{equation}
Let  $\CT(\eta)_{(i)}$ be the subset of $\Tab(\eta)$ consisting of all tableaux  whose first box filled with `2' occurs at the right-end box of the $i$-th row. By the Littlewood-Richardson condition, $\CT(\eta)_{(1)}=\emptyset$. Denote by $\CT(\eta)_{\ell(\eta)+1}$ the subset consisting of the one tableau whose boxes all contain `1', where $\ell(\eta)$ is the number of rows of $\eta$.
We have a disjoint union of $\Tab(\eta)$ and a formula of $\zeta(\Tab(\eta))$:
$$
\Tab(\eta)=\coprod^{\ell(\eta)+1}_{i=2} \CT(\eta)_{(i)}
\text{ and }
\zeta(\Tab(\eta))=1+\sum^{\ell(\eta)}_{i=2}\zeta(\CT(\eta)_{(i)}).
$$
In order to prove~\eqref{eq:zeta}, we only need to consider the rows whose numbers of boxes are non-zero.

In order to evaluate $\zeta(\CT(\eta)_{(i)})$ for $i\geq 2$, we reduce to a skew diagram $\eta^{(i)}$ of a smaller size than $\eta$.
Let $\eta^{(i)}$ be the horizontal strip such that the numbers of rows are $(\sum^{i-1}_{j=1}\eta_{j}-1,\eta_{i}-1,\dots,\eta_{\ell(\eta)})$ from top to bottom, obtained by merging  the boxes in the top $i-1$ rows of $\eta$ to the $i$-th row, becoming the first row of $\sum^{i-1}_{j=1}\eta_{j}$ boxes,  and then removing the right end boxes in the top two rows. For instance,
$$
\eta=\begin{ytableau}
\none&\none&\none&\none & 1 &1\\
\none&\none&1 &1 &\none&\none\\
1&2
\end{ytableau}\qquad
\eta^{(2)}=\begin{ytableau}
\none &1&1&1\\
1\end{ytableau}.
$$
Then we obtain a bijection between $\CT(\eta)_{(i)}$ and $\Tab(\eta^{(i)})$,  and
$$
\zeta(\CT(\eta)_{(i)})=-\zeta(\Tab(\eta^{(i)})).
$$

Now, we have
\begin{equation}\label{eq:zeta-inductive}
\zeta(\Tab(\eta))=1-\sum^{\ell(\eta)}_{i=2}\zeta(\Tab(\eta^{(i)}))
\end{equation}
and $|\eta^{(i)}|=|\eta|-2$ for $2\leq i\leq \ell(\eta)$.
We apply this inductive formula~\eqref{eq:zeta-inductive} to prove~\eqref{eq:zeta}.
If $|\eta|=0$, $\zeta(\Tab(\eta))=1$. If $|\eta|=2$, we have two types of skew shapes, $\ydiagram{2}$ and $\ydiagram{1+1,1}$, denoted by $(2)/(0)$ and $(2,1)/(1)$ respectively. Then it is easy to check that $\zeta(\Tab((2)/(0)))=1$ and $\zeta(\Tab((2,1)/(1)))=0$.

In general, if $\eta$ is even, then $\eta^{(i)}$ is not even for all $2\leq i\leq \ell(\eta)$ and $|\eta^{(i)}|<|\eta|$. By induction, $\zeta(\Tab(\eta^{(i)}))=0$ and then $\zeta(\Tab(\eta))=1$ by~\eqref{eq:zeta-inductive}.
If instead $\eta$ is not even, let $i_{\max}$ be the maximal integer such that $\eta_{i_{\max}}$ is odd. Then $\eta^{(i)}$ is  even if and only if  $i=i_{\max}$. By induction, we have $\zeta(\Tab(\eta^{i}))=1$ when $i=i_{\max}$ and  $\zeta(\Tab(\eta^{i}))=0$ when $i\ne i_{\max}$.
Hence
$
\zeta(\Tab(\eta))=1-\zeta(\Tab(\eta^{(i_{\max})}))=0.
$
This completes the proof of~\eqref{eq:zeta} and the lemma follows.

\end{proof}

For example, when $G=\Sp_{4}$,  $\Xi_{1}=\bar{\rho}_{S}(2)-\bar{\rho}_{S}(1^{2})+\binom{1}{1}$. If $G=\Sp_{12}$, then
\begin{align*}
\Xi_{3}=&\binom{3}{3}+\binom{1,2}{1,2}+\binom{1^{3}}{1^{3}}+\binom{4}{2}+\binom{2^{2}}{2}-\binom{1^{2},2}{2}+\binom{1,3}{1^{2}}\\
&-\binom{2^{2}}{1^{2}}-\binom{1^{4}}{1^{2}}+\binom{5}{1}-\binom{1^{2},3}{1}+\binom{1,2^{2}}{1}.
\end{align*}

Let $\CS_{n}$ be a set consisting of all special symbols of rank $2n$ whose associated pairs of partitions under the map $\CL^{-1}$ are even horizontal strips, that is, $\CL^{-1}(\CS_{n})$ is the same as the subset of $\Gamma^{\circ}_{2}$ consisting of all horizontal strips of size $2n$.
Given a special symbol
$$
Z=\begin{pmatrix}\lam_{1}, \lam_{2},\cdots,\lam_{m+1}\\
\mu_1,\mu_2,\cdots,\mu_{m}
\end{pmatrix},
$$
define an admissible arrangement
$$
\Phi_{Z}=\cpair{(\lam_{i},\mu_{i})\mid \lam_{i}\ne \mu_{i}, \text{ for }  1\leq i\leq m}
$$
and a subset of $\Phi_{Z}$
$$
\hat{\Phi}_{Z}=\cpair{(\lam_{i},\mu_{i})\in \Phi\mid \lam_{i}-\mu_{i}\equiv 1\bmod 2}.
$$

\begin{thm}\label{thm:main}
$$
\CU(\Ind^{\Sp_{4n}(\BF_{q})}_{\Sp_{2n}(\BF_{q^{2}})}\triv)=\sum_{Z\in  \CS_{n}} R(c(Z,\Phi_{Z},\hat{\Phi}_{Z})).
$$
In particular, the unipotent cuspidal representation of $\Sp_{4n}(\BF_{q})$ has non-trivial $\Sp_{2n}(\BF_{q^{2}})$-invariants.
\end{thm}

\begin{proof}
First, we prove that
\begin{equation}\label{eq:Xi-proof}
\Xi_{n}=\sum_{Z\in \CS_{n}} c(Z,\Phi,\hat{\Phi}_{Z}).
\end{equation}
Assume that $\alpha/\beta$ is in $\Gamma^{\circ}_{2}$ and $|\alpha|+|\beta|=2n$. Then $\ell(\alpha)-\ell(\beta)\leq 2$.

Recall that we increase the lengths of $\alpha$ and $\beta$ such that $\ell(\alpha)=\ell(\beta)+1$ by adding zeros. Assume that $\alpha=(\alpha_{1},\alpha_{2},\dots,\alpha_{m+1})$ and $\beta=(\beta_{1},\beta_{2},\dots,\beta_{m})$ with $\alpha_{i}\leq \alpha_{i+1}$ and $\beta_{i}\leq \beta_{i+1}$ and  at least one of  $\alpha_{1}$ and $\beta_{1}$ is nonzero.

Since $\beta<\alpha$, we have $\beta_{i}\leq \alpha_{i+1}$ for $1\leq i\leq m$.
Set $\lam_{i}=\alpha_{i}+i-1$ for $1\leq i\leq m+1$ and $\mu_{i}=\beta_{i}+i-1$ for $1\leq i\leq m$.
Then the symbol $\Lam=\binom{\lam}{\mu}$ is of rank $2n$ and defect 1.
Since $\beta<\alpha$,  $\Lam$ is special if and only if $\alpha_{i}\leq \beta_{i}$ for $1\leq i\leq m$. Then  the skew diagram $\alpha/\beta$ has no two boxes in each column, that is, $\alpha/\beta$ is a horizontal strip and even.
Hence, $\Lam$ is special if and only if $\Lam$ is in $\CS_{n}$.

Let $\Psi$ be a subset of $\Phi_{Z}$. For each subset $\Psi\subset \Phi_{Z}$, define
$$
\Lam(\Psi)=\begin{pmatrix}
Z_{0}\coprod \Psi_{*}\coprod (\Phi-\Psi)^{*}\\
Z_{0}\coprod \Psi^{*}\coprod (\Phi-\Psi)_{*}
\end{pmatrix}
$$
 and denote by $\CL^{-1}(\Lam(\Psi))=\binom{\alpha'}{\beta'}$. Since $\max\{\alpha_{i},\beta_{i}\}\leq \min\{\alpha_{i+1},\beta_{i+1}\}$ for $1\leq i\leq m$ (when $i=m$, let $\min\{\alpha_{i+1},\beta_{i+1}\}=\alpha_{m+1}$), we have $\beta'_{i}\leq \alpha'_{i+1}$ for $1\leq i\leq m$.

We will show that for all subsets $\Psi$ of $\Phi_{Z}$
\begin{equation}\label{eq:Xi-Psi}
\apair{\Xi_{n}, (-1)^{|\Psi\cap \hat{\Phi}_{Z}|}\CL^{-1}(\Lam(\Psi))}=1.
\end{equation}
If $\Psi=\emptyset$, then $\Lam(\Psi)=Z$ and Equation~\eqref{eq:Xi-Psi} is true. We verify Equation~\eqref{eq:Xi-Psi} by induction. Assume that Equation~\eqref{eq:Xi-Psi} is true for $\Psi$. Then by adding a pair $ (\lam_{i},\mu_{i}) $ in $ \hat{\Phi}_{Z}$  but not in $\Psi$, we obtain a subset $\Psi_{1}=\Psi\cup \{(\lam_{i},\mu_{i})\}$.

By $ (\lam_{i},\mu_{i}) \notin \Psi$, we may assume
$$
\CL^{-1}(\Lam(\Psi))=\begin{pmatrix}
\alpha'_{1},\alpha'_{2},\cdots,\alpha_{i},\alpha'_{i+1},\cdots,\alpha'_{m},\alpha_{m+1}\\
\beta'_{1},\beta'_{2},\cdots,\beta'_{i-1},\beta_{i},\cdots,\beta'_{m}
\end{pmatrix}
$$
and then
$$
\CL^{-1}(\Lam(\Psi_{1}))=
\begin{pmatrix}
\alpha'_{1},\alpha'_{2},\cdots,\beta_{i},\alpha'_{i+1},\cdots,\alpha_{m+1}\\
\beta'_{1},\beta'_{2},\cdots,\beta'_{i-1},\alpha_{i},\cdots,\beta'_{m}
\end{pmatrix}.
$$
 Now let us consider the skew diagrams obtained by the differences of partitions in $\CL^{-1}(\Lam(\Psi))$ and $\CL^{-1}(\Lam(\Psi_{1}))$ respectively, which are the same expect the following two rows
$$
\CL^{-1}(\Lam(\Psi))\colon
\begin{tabular}{ c c c c}
\cline{4-4}
$\beta'_{i-1}$ &  $\alpha_{i}-\beta'_{i-1}$ & \multicolumn{1}{c|}{$\beta_{i}-\alpha_{i}$} & \multicolumn{1}{r|}{$\alpha'_{i+1}-\beta_{i}$} \\ \cline{2-2} \cline{4-4}
 \multicolumn{1}{c|}{$\beta'_{i-1}$}& \multicolumn{1}{c|}{$\alpha_{i}-\beta'_{i-1}$} & &\\ \cline{2-2}
\end{tabular}
$$
and
$$
\CL^{-1}(\Lam(\Psi_{1}))\colon
\begin{tabular}{ c c c c}
  \cline{3-4}
  $\beta'_{i-1}$ &  \multicolumn{1}{c|}{$\alpha_{i}-\beta'_{i-1}$} &  \multicolumn{1}{c|}{$\beta_{i}-\alpha_{i}$} &  \multicolumn{1}{r|}{$\alpha'_{i+1}-\beta_{i}$} \\   \cline{2-4}
   \multicolumn{1}{c|}{$\beta'_{i-1}$} &  \multicolumn{1}{c|}{$\alpha_{i}-\beta'_{i-1}$} &  \multicolumn{1}{c|}{$\beta_{i}-\alpha_{i}$} &   \\ \cline{2-3}
  \end{tabular}.
$$
Here the boxes are the two rows of  the skew diagrams associated with $\CL^{-1}(\Lam(\Psi))$ and $\CL^{-1}(\Lam(\Psi_{1}))$, and the integers in the boxes are the parts corresponding to the differences of partitions in $\CL^{-1}(\Lam(\Psi))$ and $\CL^{-1}(\Lam(\Psi_{1}))$. The integers on the left of the boxes are the parts of the smaller partitions in $\CL^{-1}(\Lam(\Psi))$ and $\CL^{-1}(\Lam(\Psi_{1}))$.
By the assumption on $\CL^{-1}(\Lam(\Psi))$ and $\alpha_{i}<\beta_{i}$, the skew diagram of $\CL^{-1}(\Lam(\Psi_{1}))$ is also in $\Gamma^{\circ}_{2}$ and 
$$
|v(\CL^{-1}(\Lam(\Psi_{1})))|=|v(\CL^{-1}(\Lam(\Psi)))|+2(\beta_{i}-\alpha_{i}).
$$
Then by Lemma~\ref{lm:Xi-even}
$$
\apair{\Xi_{n},\CL^{-1}(\Lam(\Psi_{1}))}_{W_{2n}}=(-1)^{|\Psi\cap \hat{\Phi}_{Z}|+\beta_{i}-\alpha_{i}}.
$$
In addition, by the definition of $\Phi_{Z}$, $(-1)^{|\Psi_{1}\cap \hat{\Phi}_{Z}|}=(-1)^{|\Psi\cap \hat{\Phi}_{Z}|+\beta_{i}-\alpha_{i}}$. Therefore
$$
\apair{\Xi_{n},\CL^{-1}(\Lam(\Psi_{1}))}_{W_{2n}}=(-1)^{|\Psi_{1}\cap \hat{\Phi}_{Z}|}.
$$

On the other hand, if there is a partition $\binom{\alpha'}{\beta'}$ in $\Gamma^{\circ}_{2}$, one may reverse the previous operation (i.e. removing pairs) and obtain a partition in $\CL^{-1}(\CS_{n})$.

In sum, $\sum_{Z\in \CS_{n}} c(Z,\Phi_{Z},\hat{\Phi}_{Z})$ is a summand of $\Xi_{n}$, and every irreducible $W_{2n}$-module in $\Xi_{n}$ is in $\sum_{Z\in \CS_{n}} c(Z,\Phi,\hat{\Phi}_{Z})$ with the same signature. Then Equation~\eqref{eq:Xi-proof} follows.

Let $d$ be the number of singles in $Z$.
By $\langle\Xi_{n}, c(Z,\Phi_{Z},\hat{\Phi}_{Z})\rangle_{W_{2n}}=2^{d}$,
$$
\langle R(c(Z,\Phi_{Z},\hat{\Phi}_{Z})),\Ind^{G^{F}}_{H^{F}}\triv\rangle_{G^{F}}=2^{d}.
$$
In addition, $\langle R(c(Z,\Phi_{Z},\hat{\Phi}_{Z})), R(c(Z,\Phi_{Z},\hat{\Phi}_{Z})) \rangle_{G^{F}}=2^{d}$. By Theorem~\ref{thm:uniq}, every unipotent representation in $R(c(Z,\Phi_{Z},\hat{\Phi}_{Z}))$ is distinguished. This completes the theorem.
\end{proof}

\begin{ex}
Let $G^{F}=\Sp_{4}(\BF_{q})$. We continue Example~\ref{ex:Sp4}. In this case, $\CS_{1}=\{\binom{2}{-},\binom{0,2}{1}\}$. Let $Z=\binom{0,2}{1}$ and  $\Phi_{Z}=\{(0,1)\}$. By~\eqref{eq:Sp4} and Theorem~\ref{thm:main},
$$
\CU(\Ind^{\Sp_{4n}(\BF_{q})}_{\Sp_{2n}(\BF_{q^{2}})}\triv)=\triv+\rho\binom{0,1}{2}+\theta_{10}.
$$
\end{ex}

\begin{ex}
Let $G^{F}=\Sp_{12}$ and  
\begin{align*}
\CS_{6}=&\left\{
\binom{0,4}{3}, \binom{0,2,4}{1,3}, \binom{0,2,3,4}{1,2,3},\binom{0,5}{2},\right.\\
&\left.
\binom{2,3}{2}, \binom{0,2,5}{1,2}, \binom{0,6}{1},\binom{6}{-} \right\}.
\end{align*}
Let $Z=\binom{0,5}{2}$ and $\Phi_{Z}=\{(0,2)\}$ and $\hat{\Phi}_{Z}=\emptyset$. We have
$$
R(c(Z,\Phi_{Z},\emptyset))=R\binom{0,5}{2}+R\binom{2,5}{0}=\rho\binom{0,5}{2}+\rho\binom{0,2}{5}.
$$
Hence these two unipotent representations are distinguished.

If $Z= \binom{0,2,4}{1,3}$ then $\Phi_{Z}=\{(0,1),(2,3)\}$ and $\hat{\Phi}_{Z}=\Phi_{Z}$.
We have
$$
R(c(Z,\Phi_{Z},\Phi_{Z}))=R\binom{0,2,4}{1,3}-R\binom{1,2,4}{0,3}-R\binom{0,3,4}{1,2}+R\binom{1,3,4}{0,2}.
$$
The decomposition of $R(c(Z,\Phi_{Z},\Phi_{Z}))$ is given in~\eqref{eq:Sp12} and each constituent is distinguished.

\end{ex}

\end{document}